\documentclass[11pt]{article}
\usepackage{amsmath}
\usepackage{amsfonts}
\usepackage{amssymb}
\usepackage{mathrsfs}
\usepackage{amsthm}
\usepackage{palatino}
\usepackage[margin=2.5cm, vmargin={1.5cm}]{geometry}

\usepackage{dcolumn}

\usepackage{times}
\usepackage{graphicx}
\usepackage{xcolor}

\usepackage{pstricks}
\usepackage{pst-plot}

\renewcommand\footnotemark{}

\newenvironment{rcases}
  {\left.\begin{aligned}}
  {\end{aligned}\right\rbrace}

\begin{document}

\title{Limits of Jensen polynomials for  partitions and other sequences}

\author{
  Cormac ~O'Sullivan\footnote{{\it Date:} Aug 15, 2021.
\newline \indent \ \ \
  \newline \indent \ \ \
Support for this project was provided by a PSC-CUNY Award, jointly funded by The Professional Staff Congress and The City
\newline \indent \ \ \
University of New York.}
  }

\date{}

\maketitle

\def\s#1#2{\langle \,#1 , #2 \,\rangle}

\def\H{{\mathbf{H}}}
\def\F{{\frak F}}
\def\C{{\mathbb C}}
\def\R{{\mathbb R}}
\def\Z{{\mathbb Z}}
\def\Q{{\mathbb Q}}
\def\N{{\mathbb N}}
\def\G{{\Gamma}}
\def\GH{{\G \backslash \H}}
\def\g{{\gamma}}
\def\L{{\Lambda}}
\def\ee{{\varepsilon}}
\def\K{{\mathcal K}}
\def\Re{\mathrm{Re}}
\def\Im{\mathrm{Im}}
\def\PSL{\mathrm{PSL}}
\def\SL{\mathrm{SL}}
\def\Vol{\operatorname{Vol}}
\def\lqs{\leqslant}
\def\gqs{\geqslant}
\def\sgn{\operatorname{sgn}}
\def\res{\operatornamewithlimits{Res}}
\def\li{\operatorname{Li_2}}
\def\lip{\operatorname{Li}'_2}
\def\pl{\operatorname{Li}}

\def\clp{\operatorname{Cl}'_2}
\def\clpp{\operatorname{Cl}''_2}
\def\farey{\mathscr F}

\def\dm{{\mathcal A}}
\def\ov{{\overline{p}}}
\def\ja{{K}}

\newcommand{\stira}[2]{{\genfrac{[}{]}{0pt}{}{#1}{#2}}}
\newcommand{\stirb}[2]{{\genfrac{\{}{\}}{0pt}{}{#1}{#2}}}
\newcommand{\norm}[1]{\left\lVert #1 \right\rVert}

\newcommand{\e}{\eqref}


\newtheorem{theorem}{Theorem}[section]
\newtheorem{lemma}[theorem]{Lemma}
\newtheorem{prop}[theorem]{Proposition}
\newtheorem{conj}[theorem]{Conjecture}
\newtheorem{cor}[theorem]{Corollary}
\newtheorem{assume}[theorem]{Assumptions}
\newtheorem{adef}[theorem]{Definition}


\newcounter{counrem}
\newtheorem{remark}[counrem]{Remark}

\renewcommand{\labelenumi}{(\roman{enumi})}
\newcommand{\spr}[2]{\sideset{}{_{#2}^{-1}}{\textstyle \prod}({#1})}
\newcommand{\spn}[2]{\sideset{}{_{#2}}{\textstyle \prod}({#1})}

\numberwithin{equation}{section}

\let\originalleft\left
\let\originalright\right
\renewcommand{\left}{\mathopen{}\mathclose\bgroup\originalleft}
\renewcommand{\right}{\aftergroup\egroup\originalright}

\bibliographystyle{alpha}

\begin{abstract}
It was discovered in \cite{GORZ} that the Jensen polynomials associated to many sequences have Hermite polynomial limits. We develop this theory  in detail, based on the log-polynomial property    which is a refinement of log-concavity and log-convexity. Applications to various partition sequences are given. An application to the sequence of factorials leads naturally  to evaluating limits of generalized Laguerre polynomials.
\end{abstract}

\section{Introduction}
A sequence of real numbers $\alpha(0),$ $\alpha(1),$ $\alpha(2), \dots$  is  log-concave if
\begin{equation}\label{logc}
\alpha(n+1)^2 \gqs \alpha(n) \cdot \alpha(n+2)
\end{equation}
 holds for $n \gqs 0$. If  the sequence is positive then being log-concave is equivalent to the sequence $\beta(n)=\log \alpha(n)$ being concave:
\begin{equation}\label{c}
2\beta(n+1) \gqs \beta(n) + \beta(n+2).
\end{equation}
Many combinatorial and number theoretic sequences are log-concave, such as the binomial coefficients $\binom{m}{n}$ as $n$ varies, for example.
Reversing the inequalities in \e{logc} and \e{c} defines log-convex and convex sequences, respectively.  Strictly log-concave means \e{logc}  holds with a strict inequality and similarly in the other cases. If the inequalities hold only for $n$ large enough then we say the sequence has the property asymptotically.

The following result mimics  the second derivative  test  and is an easy exercise.

\begin{prop} \label{first}
Let $\alpha(n)$ be a sequence of positive real numbers.  Suppose there exist $\delta(n)> 0$, $A(n)$  and $\kappa = \pm 1$  so that
\begin{equation} \label{ac2pp}
    \log\left(\frac{\alpha(n+j)}{\alpha(n)}\right) = A(n) j +\kappa \cdot \delta(n)^2 j^2 +  o\left(\delta(n)^2 \right)
  \end{equation}
as $n \to \infty$ for $j=1, 2$. Then the sequence $\alpha$ is asymptotically strictly log-concave if $\kappa=-1$ and asymptotically strictly log-convex if $\kappa=1$.
\end{prop}

For another approach to log-concavity, note that \e{logc} is equivalent to the polynomial
$$
J_\alpha^{2,n}(X):=\binom{2}{0} \alpha(n) + \binom{2}{1} \alpha(n+1)X +\binom{2}{2} \alpha(n+2)X^2
$$
having  real roots. 
Define the $d$th Jensen polynomial associated to the sequence  $\alpha$, and with shift $n$, as
\begin{equation} \label{jdef}
   J_\alpha^{d,n}(X):= \sum_{j=0}^d \binom{d}{j} \alpha(n+j)X^j.
\end{equation}
 For any $d \in \Z_{\gqs 2}$ then $J_\alpha^{d,n}(X)$ having all real zeros   implies \e{logc} by a result of Newton; see \cite[p. 53]{HLP}.
We are also interested in the reciprocal polynomials
\begin{equation} \label{kdef}
  \ja_\alpha^{d,n}(X):= X^d J_\alpha^{d,n}(1/X) = \sum_{j=0}^d \binom{d}{j} \alpha(n+j)X^{d-j}.
\end{equation}
In fact it was the  $K_\alpha^{d,0}(X)$ polynomials that  Jensen used in \cite{Je13}  to give criteria for an entire function to have only real zeros. Shifting by $n$ corresponds to working with the $n$th derivative of the function; see for example \cite[Sect. 3]{OS21}. 

\begin{prop} \label{basic}
Let $\alpha(n)$ be a sequence of positive real numbers and let $d$ be a positive integer.  Suppose there exists $A(n)$    so that
\begin{equation*} 
    \log\left(\frac{\alpha(n+j)}{\alpha(n)}\right) = A(n) j +  o\left(1 \right)
  \end{equation*}
as $n \to \infty$ for $j=1, 2, \dots ,d$. Then, as $n\to \infty$,
\begin{align}
    \frac{1}{\alpha(n)} J^{d,n}_\alpha \left( \frac{X-1}{\exp(A(n))}\right) & \to X^d, \label{limj}\\
 \frac{\exp(A(n))^{-d}}{\alpha(n)} K^{d,n}_\alpha \left( \frac{X-1}{\exp(-A(n))}\right) & \to
 X^d. \label{limk}
\end{align}
\end{prop}
\begin{proof}
This is straightforward using the definitions \e{jdef}, \e{kdef} and that $\exp(o(1))=1+o(1)$.
\end{proof}

Throughout this paper, limits such as \e{limj}, \e{limk}, and  \e{onno}, \e{hjak} below,  mean that the  coefficients of the polynomial on the left tend to the corresponding coefficients  on the right.
The following definition, based on \cite[Eq. (15)]{GORZ},  describes a situation when we have more information about the sequence $\alpha$ than Propositions \ref{first}, \ref{basic} and   $\log \alpha(n+j)$ can be well approximated by a polynomial in $j$ for large $n$.

\begin{adef} \label{dbf}
{\rm A sequence of positive real numbers $\alpha(n)$ is {\em log-polynomial of degree $m \gqs 2$, with data $\{A(n),\kappa,\delta(n)\}$} for $\kappa=\pm 1$, if it satisfies the following conditions.
 There exist sequences $A(n)$, $\delta(n)$ and $g_k(n)$ for $k=3,4, \dots, m$  so that
\begin{equation} \label{hjk}
    \log\left(\frac{\alpha(n+j)}{\alpha(n)}\right) = A(n) j +\kappa \cdot \delta(n)^2 j^2 + \sum_{k=3}^{m} g_k(n) j^k + o\left(\delta(n)^{m+1}\right)
  \end{equation}
as $n \to \infty$, for $j=1,2, \dots, m+1$. We also require  $\delta(n)>0$, $\delta(n) \to 0$ and $g_k(n)=o(\delta(n)^k)$ as $n \to \infty$.
}
\end{adef}

Some basic properties of log-polynomial sequences are contained in the next result.

\begin{prop} \label{logp}
Suppose  $\alpha(n)$ is a log-polynomial sequence of degree $m$ with data $\{A(n),\kappa,\delta(n)\}$.
\begin{enumerate}
  \item  Then $\alpha(n)$ is  log-polynomial  for all degrees $2,3, \dots, m$.
  \item The reciprocal sequence $1/\alpha(n)$ is  log-polynomial of degree $m$ with data $\{-A(n),-\kappa,\delta(n)\}$.
  \item If $\alpha(n)$ is also log-polynomial of degree $m$ with different data $\{A^*(n),\kappa^*,\delta^*(n)\}$ then $\kappa^*=\kappa$ and
\begin{equation} \label{ophf}
   A^*(n)  = A(n)+o\left(\delta(n)^{m+1}\right),  \qquad
\delta^*(n)^2  = \delta(n)^2 + o\left(\delta(n)^{m+1}\right)
\qquad \text{as \quad $n\to \infty$.}
\end{equation}
\end{enumerate}
\end{prop}


Let $H_d(X)$ be the $d$th Hermite polynomial. As reviewed in section \ref{polys}, we have
\begin{align} \label{mv1}
   H_d(X/2) & = \sum_{k=0}^d \frac{d!}{(d-k)!} \binom{d-k}{k} (-1)^{k} X^{d-2k}, \\
i^{-d} H_d(i X/2) &= \sum_{k=0}^d \frac{d!}{(d-k)!} \binom{d-k}{k} X^{d-2k}. \label{mv2}
\end{align}
Griffin, Ono, Rolen and Zagier in \cite{GORZ} discovered a refinement of Proposition \ref{basic}, showing   that in some cases the $d$th Jensen polynomial of a sequence has limit  $H_d(X)$ as the shift $n$ increases. Precisely, Theorem 3 of \cite{GORZ}   says that if $\alpha(n)$ is a log-polynomial sequence of degree $d$ with data $\{A(n),-1,\delta(n)\}$ then
\begin{equation} \label{onno}
   \frac{\delta(n)^{-d}}{\alpha(n)} J^{d,n}_\alpha \left( \frac{\delta(n) X -1}{\exp( A(n))}\right)  \to H_d(X/2) \qquad \text{as} \quad n\to \infty.
\end{equation}
(Unfortunately \cite[Thm. 3]{GORZ} is slightly misstated, omitting the sum over $k$ in \e{hjk}; it is given correctly in \cite[Eq. (15)]{GORZ}.) Replacing $X$ by $X/\delta(n)$ in \e{onno} implies that 
\begin{equation} \label{onno2}
  \frac{1}{\alpha(n)} J^{d,n}_\alpha \left( \frac{X-1}{\exp(A(n))}\right)
 \approx X^d -2 \binom{d}{2} \delta(n)^2 X^{d-2}+ 12 \binom{d}{4} \delta(n)^4 X^{d-4}+ \cdots,
\end{equation}
where \e{onno2} means that given any $\varepsilon>0$ the  coefficients of  $X^{d-j}$ on both sides agree to within $\varepsilon \cdot \delta(n)^j$ for $n$ sufficiently large.

 The next definition allows us to include the natural extensions of \e{onno} to  the reciprocal Jensen polynomials $\ja^{d,n}_\alpha$ and also the $\kappa=1$ case.

\begin{adef} \label{hjdef}
{\rm A sequence of positive real numbers $\alpha(n)$ is {\em Hermite-Jensen of degree $d$ and data $\{A(n)$, $\kappa$, $\delta(n)\}$} if it satisfies the following conditions. As $n\to \infty$,
\begin{equation} \label{hjak}
\begin{rcases}
   \frac{\delta(n)^{-d}}{\alpha(n)} J^{d,n}_\alpha \left( \frac{\delta(n) X -1}{\exp( A(n))}\right) & \\
   \frac{(\exp(A(n)) \delta(n))^{-d}}{\alpha(n)}
\ja^{d,n}_\alpha \left( \frac{\delta(n) X -1}{\exp(-A(n))}\right)
 &
\end{rcases}
 \to \begin{cases}
H_d(X/2) & \text{ if $\kappa=-1$,} \\
i^{-d} H_d(i X/2) & \text{ if $\kappa=1$.}
\end{cases}
\end{equation}
}
\end{adef}

As \e{hjak} is always true for $d=0$ we restrict our attention to degrees $d \gqs 1$. Wagner in \cite[Def. 1.1]{Wa} uses the term `Hermite-Jensen' to  refer to what we are calling log-polynomial sequences, though with the sum over $k$ in \e{hjk} missing.

\begin{theorem} \label{vog}
Suppose $\alpha(n)$ is a log-polynomial  sequence of degree $m$ with data $\{A(n)$, $\kappa$, $\delta(n)\}$. Then this sequence is Hermite-Jensen of degree $d$ for $1\lqs d \lqs m+1$ and the same data.
\end{theorem}

It is perhaps surprising that in \e{hjak} we also find Hermite polynomials as the limit of the reciprocal polynomials $\ja^{d,n}_\alpha$. The Hermite polynomials themselves do not seem to satisfy any reciprocal relations.

As the orthogonal polynomials $H_d(x)$ have distinct real roots, we will see that it follows from  \e{hjak}   that $J^{d,n}_\alpha(X)$   must also have distinct real roots for $n$ large enough when $\kappa=-1$.
In this way, Theorem \ref{vog} may be used to show that the Jensen polynomials associated to many sequences must have distinct real roots for $n$ large.   This idea was introduced in \cite{GORZ}. An application to the partitions sequence  $p(n)$ in Theorem 5 of \cite{GORZ}   proved the real roots  Conjecture 1.5 of \cite{CJW}. \cite[Thm. 7]{GORZ} generalized this to sequences of coefficients of weakly holomorphic modular forms and a further application to the Taylor coefficients of the Riemann zeta $\zeta(s)$ at $s=1/2$ was given  in \cite[Thm. 2]{GORZ}.

The precise properties of Jensen polynomials associated to log-polynomial sequences may be stated as follows:

\begin{theorem} \label{redct}
Assume the sequence $\alpha(n)$ is log-polynomial for some degree $m\gqs 2$.
\begin{enumerate}
  \item If $\kappa=-1$ then the sequence is asymptotically strictly log-concave. If $\kappa=1$ then the sequence is asymptotically strictly log-convex.
  \item Fix $d$ with $1\lqs d \lqs m+1$.  For $n$ sufficiently large, the zeros of $J_\alpha^{d,n}(X)$ and $K_\alpha^{d,n}(X)$ can be described as follows. Firstly, they are all distinct. If $\kappa=-1$ they are all real. If $\kappa=1$ they are all not real, except for a single real zero when $d$ is odd.
\end{enumerate}
\end{theorem}

In just a few pages,  the authors of \cite{GORZ} introduce many striking results and techniques. Our aim in this article is to expand and clarify these ideas, supply some omitted details, and to give further applications.

\section{Applications to partitions}
In sections \ref{lim}, \ref{main} we provide  detailed proofs  of Theorems \ref{vog} and \ref{redct}.
It is mentioned in \cite[Sect. 6]{GORZ} that  \e{onno} is still valid if $A(n)$ and $\delta(n)$ are replaced by approximations. The precise conditions needed are given next.

\begin{theorem} \label{vog2}
Let $\alpha(n)$ be a log-polynomial  sequence of degree $m$ with data $\{A(n)$, $\kappa$, $\delta(n)\}$.
Suppose $A^*(n)$ and $\delta^*(n)$ satisfy
\begin{equation} \label{tre}
   A^*(n) = A(n)+o(\delta(n)), \qquad \delta^*(n) = \delta(n)(1+o(1)) \qquad \text{as \quad $n\to \infty$.}
\end{equation}
Then $\alpha(n)$ is Hermite-Jensen of degree $d$ for $1\lqs d \lqs m+1$ with  data $\{A^*(n)$, $\kappa$, $\delta^*(n)\}$.
\end{theorem}

Comparing \e{ophf} and \e{tre} shows that the data required for a sequence to be Hermite-Jensen can be much less precise than the data required for it to be log-polynomial.

Theorems \ref{vog} and \ref{vog2} are illustrated next with applications to various partition functions.
  A partition of $n$ is a  non increasing sequence of positive integers that  sum to $n$. The partition function $p(n)$ counts all the partitions of $n$. The number of overpartitions of $n$ is $\overline p(n)$ and these are partitions where the first appearance of a part of each size may or may not be overlined. For $k\in \Z_{\gqs 2}$, let $b_k(n)$ denote the number of partitions of $n$ where no part sizes are multiples of $k$. These are called $k$-regular partitions.  Glaisher's theorem of 1883 implies that $b_k(n)$ also counts the number of partitions of $n$ where parts appear at most $k-1$ times. Hence $b_2(n)$ gives the number of partitions of $n$ into distinct parts.

The well-known generating functions for $p(n)$, $\overline p(n)$ and $b_k(n)$ are
\begin{equation}\label{pgf}
  \sum_{n=0}^\infty p(n) q^n =  \prod_{j=1}^\infty \frac 1{1-q^j}, \qquad
\sum_{n=0}^\infty \overline p(n) q^n =  \prod_{j=1}^\infty \frac {1+q^j}{1-q^j}, \qquad \sum_{n=0}^\infty b_k(n) q^n =  \prod_{j=1}^\infty \frac {1-q^{k j}}{1-q^j}.
\end{equation}
Multiplying the first product in \e{pgf} by the third when $k=2$ gives the second. This implies the relation
\begin{equation*}
  \overline p(n) =\sum_{r=0}^n p(r)\cdot b_2(n-r).
\end{equation*}

Precise asymptotics are available for these partition functions. As $n \to \infty$,
\begin{align}\label{hard2}
 p(n) & =\frac{1}{4\sqrt{3}n'} \left(1-\frac{1}{ \sqrt{\g n'}}\right) e^{ \sqrt{\g n'}} \left( 1+O(e^{- \sqrt{\g n}/2}) \right), \\
  \overline p(n) & =\frac{1}{8n} \left(1-\frac{1}{\pi \sqrt{n}}\right) e^{\pi \sqrt{n}} \left( 1+O(e^{-2\pi \sqrt{n}/3}) \right), \label{ohard2}\\
 b_k(n) & = \frac{\sqrt{k'}}{2 \sqrt{k \cdot n''}} I_1\left( \sqrt{k' \cdot n''}\right)
\left( 1+O\left( e^{-\varepsilon_k \sqrt{n}}\right)\right), \label{bkas}
\end{align}
for certain $\varepsilon_k>0$. We are using the notation
\begin{equation} \label{ntn}
   \g:=\frac{2\pi^2}{3}, \qquad  k':=\frac{2\pi^2}{3}\left(1-\frac 1k \right), \qquad n':=n-\frac{1}{24}, \qquad n'':=n+\frac{k-1}{24}.
\end{equation}
Also $I_1$ is a  modified Bessel function of the first kind \cite[Eq. (4.12.2)]{AAR}:
\begin{equation}\label{ibes}
  I_\alpha(2z):= z^\alpha \sum_{j=0}^\infty \frac{z^{2j}}{ \G(\alpha+j+1) j!}.
\end{equation}
The estimate \e{hard2} follows from the work of Hardy and Ramanujan in \cite{HR} and is given in this form by Rademacher in  \cite[p. 278]{Ra}. Then \e{ohard2} follows similarly, is mentioned in \cite{HR}, and generalized in \cite{Sil}. The estimate \e{bkas} is Corollary 4.1 of \cite{Hag}.


\begin{theorem} \label{partit}
With $S:= \pi (2/3)^{1/2}(n-1/24)^{1/2}$, set 
\begin{equation} \label{part-ad}
  A(n):=\frac{\pi^2}{3} \left(\frac{1}{S-1} - \frac{3}{S^2} \right), \qquad \delta(n):= \frac{\pi^2}{3} \left(\frac{1}{2S(S-1)^2} - \frac{3}{S^4} \right)^{1/2}.
\end{equation}
Then for every  degree $m \gqs 2$ the partition sequence $p(n)$ is log-polynomial with data $\{A(n)$, $-1$, $\delta(n)\}$. Hence $p(n)$ is also Hermite-Jensen for all positive integer degrees and the same data or, more simply, $\{A^*(n)$, $-1$, $\delta^*(n)\}$ for
\begin{equation} \label{part-ads}
   A^*(n):=\frac{\pi}{6^{1/2}n^{1/2}}, \qquad  \delta^*(n) := \frac{\sqrt{\pi}}{2 \cdot 6^{1/4} n^{3/4}}.
\end{equation}
\end{theorem}

Theorem \ref{partit} is proved in section \ref{p-proof}. It agrees with the results for $p(n)$ in sections 3 and 6 of \cite{GORZ}, giving a more precise statement. See also \cite[Eq. (1.1)]{Lar}, though this is stated incorrectly.
Theorem \ref{opartit} gives the corresponding result for  overpartitions.
For $k$-regular partitions we have

\begin{theorem} \label{kreg}
Fix an integer $k\gqs 2$ and let $G(z)$ be the Bessel ratio $I_0(z)/I_1(z)$. With \e{ntn}, set
\begin{equation} \label{mess}
   A(n):=\frac{1}2 \sqrt{\frac{k'}{n''}}G\left( \sqrt{k' n''}\right) -\frac 1{n''}, \qquad  \delta(n) := \frac{1}{2n''}\left( \frac{k' n''}{2}\left(G\left( \sqrt{k' n''}\right)^2 -1 \right) - 2\right)^{1/2}.
\end{equation}
Then for every  degree $m \gqs 2$ the $k$-regular partition sequence $b_k(n)$ is log-polynomial with data $\{A(n)$, $-1$, $\delta(n)\}$. Hence $b_k(n)$ is also Hermite-Jensen for all positive integer degrees and the same data or, more simply, $\{A^*(n)$, $-1$, $\delta^*(n)\}$ for
\begin{equation} \label{pgb}
   A^*(n):=\frac{\pi (1-1/k )^{1/2}}{6^{1/2} n^{1/2}}, \qquad  \delta^*(n) := \frac{\sqrt{\pi}(1-1/k )^{1/4}}{2 \cdot 6^{1/4} n^{3/4}}.
\end{equation}
\end{theorem}

Theorem \ref{kreg} extends and clarifies the main result of  \cite{CP} as discussed in section \ref{kregpa}.   Theorem \ref{redct} then implies

\begin{cor} \label{krak}
The sequences $p(n)$, $\overline p(n)$ and $b_k(n)$ are asymptotically strictly log-concave.
For fixed  $d$, the Jensen polynomials $J^{d,n}_{p}(X)$, $J^{d,n}_{\overline p}(X)$ and $J^{d,n}_{b_k}(X)$ each have distinct real roots for all $n$ sufficiently large.
\end{cor}

The asymptotic strict log-concavity of these sequences may also be shown more directly from \e{hard2}, \e{ohard2} and \e{bkas}, reducing to the fact that $\sqrt{n}$ is concave. By using more precise asymptotic expansions, it is proved in \cite{Ni78} and \cite{DP} that $p(n)$ is log-concave (i.e. \e{logc} holds) exactly for $n\gqs 25$, and shown in \cite{Eng} that $\overline p(n)$ is  log-concave for all $n$.  The precise $n$s
 for which $J^{d,n}_{p}(X)$ has real roots are given in \cite{CJW} for $d=3$ and also in \cite{Lar} for $d=3,4,5$.


\SpecialCoor
\psset{griddots=5,subgriddiv=0,gridlabels=0pt}
\psset{xunit=1cm, yunit=0.0003125cm}
\psset{linewidth=1pt}
\psset{dotsize=2pt 0,dotstyle=*}

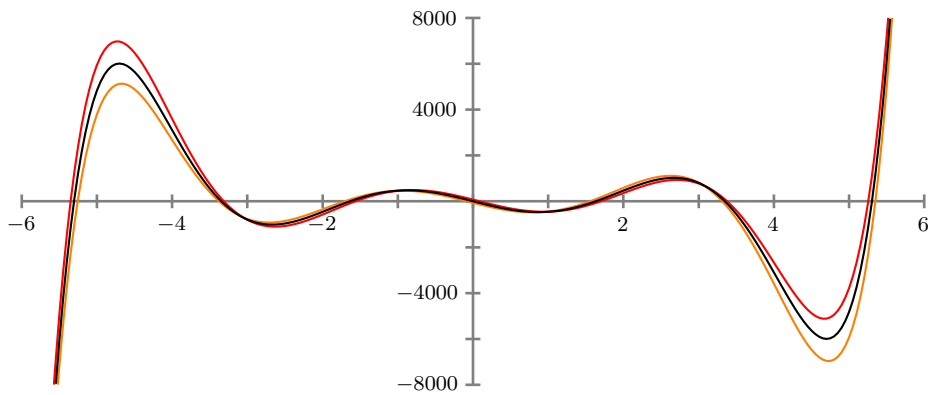
\begin{figure}[ht]
\begin{center}
\begin{pspicture}(-6,-8000)(6,8000) 

\psline[linecolor=gray](-6,0)(6,0)
\multirput(-6,-300)(1,0){13}{\psline[linecolor=gray](0,0)(0,600)}
\psline[linecolor=gray](0,-8000)(0,8000)
\multirput(-0.1,-8000)(0,2000){9}{\psline[linecolor=gray](0,0)(0.2,0)}

\rput(-6,-1000){$_{-6}$}
\rput(-4,-1000){$_{-4}$}
\rput(-2,-1000){$_{-2}$}
\rput(6,-1000){$_{6}$}
\rput(4,-1000){$_{4}$}
\rput(2,-1000){$_{2}$}

\rput(-0.6,-8000){$_{-8000}$}
\rput(-0.6,-4000){$_{-4000}$}
\rput(-0.5,8000){$_{8000}$}
\rput(-0.5,4000){$_{4000}$}

\savedata{\mydata}[
{{-5.5198, -7998.57}, {-5.5, -7217.57}, {-5.45, -5389.18}, {-5.4, -3755.6}, {-5.35,
-2303.37}, {-5.3, -1019.65}, {-5.25, 107.776}, {-5.2,
  1090.51}, {-5.15, 1939.57}, {-5.1, 2665.43}, {-5.05, 3277.99}, {-5.,
   3786.64}, {-4.95, 4200.26}, {-4.9, 4527.24}, {-4.85,
  4775.51}, {-4.8, 4952.53}, {-4.75, 5065.32}, {-4.7, 5120.5}, {-4.65,
   5124.25}, {-4.6, 5082.41}, {-4.55, 5000.4}, {-4.5, 4883.3}, {-4.45,
   4735.87}, {-4.4, 4562.49}, {-4.35, 4367.27}, {-4.3,
  4153.99}, {-4.25, 3926.17}, {-4.2, 3687.02}, {-4.15, 3439.5}, {-4.1,
   3186.34}, {-4.05, 2930.01}, {-4., 2672.76}, {-3.95,
  2416.61}, {-3.9, 2163.4}, {-3.85, 1914.77}, {-3.8, 1672.15}, {-3.75,
   1436.83}, {-3.7, 1209.93}, {-3.65, 992.395}, {-3.6,
  785.05}, {-3.55, 588.569}, {-3.5, 403.505}, {-3.45, 230.29}, {-3.4,
  69.243}, {-3.35, -79.417}, {-3.3, -215.568}, {-3.25, -339.178},
{-3.2, -450.294}, {-3.15, -549.039}, {-3.1, -635.604}, {-3.05,
-710.243}, {-3., -773.266}, {-2.95, -825.032}, {-2.9, -865.945},
{-2.85, -896.45}, {-2.8, -917.027}, {-2.75, -928.183}, {-2.7,
-930.452}, {-2.65, -924.39}, {-2.6, -910.566}, {-2.55, -889.566},
{-2.5, -861.983}, {-2.45, -828.413}, {-2.4, -789.459}, {-2.35,
-745.719}, {-2.3, -697.789}, {-2.25, -646.258}, {-2.2, -591.704},
{-2.15, -534.696}, {-2.1, -475.786}, {-2.05, -415.513}, {-2.,
-354.397}, {-1.95, -292.936}, {-1.9, -231.611}, {-1.85, -170.876},
{-1.8, -111.165}, {-1.75, -52.8849}, {-1.7, 3.58336}, {-1.65,
  57.8852}, {-1.6, 109.693}, {-1.55, 158.709}, {-1.5,
  204.662}, {-1.45, 247.31}, {-1.4, 286.441}, {-1.35, 321.871}, {-1.3,
   353.445}, {-1.25, 381.038}, {-1.2, 404.553}, {-1.15,
  423.92}, {-1.1, 439.1}, {-1.05, 450.078}, {-1., 456.867}, {-0.95,
  459.506}, {-0.9, 458.061}, {-0.85, 452.618}, {-0.8,
  443.291}, {-0.75, 430.213}, {-0.7, 413.539}, {-0.65,
  393.446}, {-0.6, 370.129}, {-0.55, 343.798}, {-0.5,
  314.685}, {-0.45, 283.031}, {-0.4, 249.096}, {-0.35,
  213.148}, {-0.3, 175.469}, {-0.25, 136.349}, {-0.2,
  96.0859}, {-0.15, 54.9837}, {-0.1,
  13.3516}, {-0.05, -28.4977}, {0., -70.2504}, {0.05, -111.592},
{0.1, -152.21}, {0.15, -191.796}, {0.2, -230.042}, {0.25, -266.652},
{0.3, -301.333}, {0.35, -333.805}, {0.4, -363.794}, {0.45, -391.043},
{0.5, -415.307}, {0.55, -436.354}, {0.6, -453.97}, {0.65, -467.961},
{0.7, -478.148}, {0.75, -484.376}, {0.8, -486.508}, {0.85, -484.431},
{0.9, -478.058}, {0.95, -467.323}, {1., -452.188}, {1.05, -432.64},
{1.1, -408.694}, {1.15, -380.393}, {1.2, -347.809}, {1.25, -311.042},
{1.3, -270.222}, {1.35, -225.509}, {1.4, -177.092}, {1.45, -125.192},
{1.5, -70.06}, {1.55, -11.9753}, {1.6, 48.7512}, {1.65,
  111.779}, {1.7, 176.74}, {1.75, 243.235}, {1.8, 310.839}, {1.85,
  379.099}, {1.9, 447.536}, {1.95, 515.647}, {2., 582.904}, {2.05,
  648.758}, {2.1, 712.638}, {2.15, 773.956}, {2.2, 832.105}, {2.25,
  886.464}, {2.3, 936.401}, {2.35, 981.271}, {2.4, 1020.42}, {2.45,
  1053.2}, {2.5, 1078.94}, {2.55, 1097.}, {2.6, 1106.71}, {2.65,
  1107.44}, {2.7, 1098.56}, {2.75, 1079.44}, {2.8, 1049.49}, {2.85,
  1008.14}, {2.9, 954.852}, {2.95, 889.113}, {3., 810.457}, {3.05,
  718.457}, {3.1, 612.739}, {3.15, 492.985}, {3.2, 358.937}, {3.25,
  210.405}, {3.3,
  47.2748}, {3.35, -130.488}, {3.4, -322.829}, {3.45, -529.599},
{3.5, -750.544}, {3.55, -985.301}, {3.6, -1233.39}, {3.65, -1494.19},
{3.7, -1766.97}, {3.75, -2050.84}, {3.8, -2344.74}, {3.85, -2647.47},
{3.9, -2957.64}, {3.95, -3273.7}, {4., -3593.89}, {4.05, -3916.23},
{4.1, -4238.55}, {4.15, -4558.44}, {4.2, -4873.26}, {4.25, -5180.1},
{4.3, -5475.82}, {4.35, -5756.97}, {4.4, -6019.83}, {4.45, -6260.37},
{4.5, -6474.25}, {4.55, -6656.79}, {4.6, -6802.99}, {4.65, -6907.46},
{4.7, -6964.45}, {4.75, -6967.83}, {4.8, -6911.05}, {4.85, -6787.14},
{4.9, -6588.7}, {4.95, -6307.87}, {5., -5936.3}, {5.05, -5465.19},
{5.1, -4885.19}, {5.15, -4186.46}, {5.2, -3358.57}, {5.25, -2390.57},
{5.3, -1270.89}, {5.35, 12.6205}, {5.4, 1472.75}, {5.45,
  3122.94}, {5.5, 4977.28}, {5.55, 7050.57}, {5.5712, 7999.5}}
  ]
\dataplot[linecolor=orange,linewidth=0.8pt,plotstyle=line]{\mydata}

\savedata{\mydata}[
{{-5.5712, -7996.53}, {-5.55, -7047.7}, {-5.5, -4974.62}, {-5.45, -3120.47}, {-5.4,
-1470.48}, {-5.35, -10.5224}, {-5.3, 1272.82}, {-5.25,
  2392.34}, {-5.2, 3360.2}, {-5.15, 4187.94}, {-5.1, 4886.55}, {-5.05,
   5466.42}, {-5., 5937.42}, {-4.95, 6308.88}, {-4.9,
  6589.61}, {-4.85, 6787.96}, {-4.8, 6911.78}, {-4.75,
  6968.49}, {-4.7, 6965.03}, {-4.65, 6907.97}, {-4.6,
  6803.43}, {-4.55, 6657.18}, {-4.5, 6474.58}, {-4.45,
  6260.65}, {-4.4, 6020.06}, {-4.35, 5757.16}, {-4.3,
  5475.97}, {-4.25, 5180.22}, {-4.2, 4873.35}, {-4.15, 4558.5}, {-4.1,
   4238.58}, {-4.05, 3916.24}, {-4., 3593.88}, {-3.95,
  3273.68}, {-3.9, 2957.61}, {-3.85, 2647.41}, {-3.8,
  2344.67}, {-3.75, 2050.76}, {-3.7, 1766.89}, {-3.65,
  1494.11}, {-3.6, 1233.29}, {-3.55, 985.204}, {-3.5,
  750.444}, {-3.45, 529.497}, {-3.4, 322.726}, {-3.35,
  130.385}, {-3.3, -47.3771}, {-3.25, -210.506}, {-3.2, -359.036},
{-3.15, -493.082}, {-3.1, -612.833}, {-3.05, -718.547}, {-3.,
-810.543}, {-2.95, -889.196}, {-2.9, -954.931}, {-2.85, -1008.22},
{-2.8, -1049.56}, {-2.75, -1079.5}, {-2.7, -1098.62}, {-2.65,
-1107.5}, {-2.6, -1106.77}, {-2.55, -1097.05}, {-2.5, -1078.99},
{-2.45, -1053.24}, {-2.4, -1020.46}, {-2.35, -981.301}, {-2.3,
-936.428}, {-2.25, -886.488}, {-2.2, -832.125}, {-2.15, -773.972},
{-2.1, -712.652}, {-2.05, -648.768}, {-2., -582.912}, {-1.95,
-515.652}, {-1.9, -447.539}, {-1.85, -379.1}, {-1.8, -310.838},
{-1.75, -243.232}, {-1.7, -176.735}, {-1.65, -111.773}, {-1.6,
-48.7442}, {-1.55, 11.9833}, {-1.5, 70.0688}, {-1.45, 125.202}, {-1.4,
   177.102}, {-1.35, 225.519}, {-1.3, 270.233}, {-1.25,
  311.053}, {-1.2, 347.82}, {-1.15, 380.404}, {-1.1, 408.704}, {-1.05,
   432.649}, {-1., 452.197}, {-0.95, 467.332}, {-0.9,
  478.066}, {-0.85, 484.439}, {-0.8, 486.515}, {-0.75,
  484.382}, {-0.7, 478.154}, {-0.65, 467.966}, {-0.6,
  453.975}, {-0.55, 436.358}, {-0.5, 415.31}, {-0.45, 391.046}, {-0.4,
   363.797}, {-0.35, 333.807}, {-0.3, 301.335}, {-0.25,
  266.653}, {-0.2, 230.043}, {-0.15, 191.796}, {-0.1,
  152.211}, {-0.05, 111.592}, {0, 70.2504}, {0.05,
  28.4977}, {0.1, -13.3516}, {0.15, -54.9835}, {0.2, -96.0856},
{0.25, -136.349}, {0.3, -175.469}, {0.35, -213.147}, {0.4, -249.094},
{0.45, -283.029}, {0.5, -314.682}, {0.55, -343.795}, {0.6, -370.125},
{0.65, -393.442}, {0.7, -413.534}, {0.75, -430.207}, {0.8, -443.284},
{0.85, -452.611}, {0.9, -458.053}, {0.95, -459.498}, {1., -456.858},
{1.05, -450.068}, {1.1, -439.09}, {1.15, -423.91}, {1.2, -404.542},
{1.25, -381.027}, {1.3, -353.434}, {1.35, -321.859}, {1.4, -286.429},
{1.45, -247.299}, {1.5, -204.651}, {1.55, -158.698}, {1.6, -109.683},
{1.65, -57.8756}, {1.7, -3.57467}, {1.75, 52.8926}, {1.8,
  111.172}, {1.85, 170.881}, {1.9, 231.614}, {1.95, 292.938}, {2.,
  354.397}, {2.05, 415.511}, {2.1, 475.782}, {2.15, 534.688}, {2.2,
  591.694}, {2.25, 646.245}, {2.3, 697.773}, {2.35, 745.7}, {2.4,
  789.436}, {2.45, 828.387}, {2.5, 861.952}, {2.55, 889.532}, {2.6,
  910.528}, {2.65, 924.348}, {2.7, 930.406}, {2.75, 928.132}, {2.8,
  916.972}, {2.85, 896.391}, {2.9, 865.882}, {2.95, 824.964}, {3.,
  773.195}, {3.05, 710.169}, {3.1, 635.526}, {3.15, 548.957}, {3.2,
  450.209}, {3.25, 339.091}, {3.3, 215.479}, {3.35,
  79.3265}, {3.4, -69.3344}, {3.45, -230.381}, {3.5, -403.597},
{3.55, -588.659}, {3.6, -785.137}, {3.65, -992.479}, {3.7, -1210.01},
{3.75, -1436.91}, {3.8, -1672.21}, {3.85, -1914.82}, {3.9, -2163.45},
{3.95, -2416.65}, {4., -2672.78}, {4.05, -2930.02}, {4.1, -3186.33},
{4.15, -3439.47}, {4.2, -3686.96}, {4.25, -3926.08}, {4.3, -4153.88},
{4.35, -4367.12}, {4.4, -4562.31}, {4.45, -4735.64}, {4.5, -4883.04},
{4.55, -5000.08}, {4.6, -5082.04}, {4.65, -5123.83}, {4.7, -5120.01},
{4.75, -5064.77}, {4.8, -4951.91}, {4.85, -4774.81}, {4.9, -4526.46},
{4.95, -4199.38}, {5., -3785.67}, {5.05, -3276.91}, {5.1, -2664.24},
{5.15, -1938.27}, {5.2, -1089.08}, {5.25, -106.211}, {5.3,
  1021.36}, {5.35, 2305.23}, {5.4, 3757.63}, {5.45, 5391.37}, {5.5,
  7219.96}, {5.5197, 7997.}}
  ]
\dataplot[linecolor=red,linewidth=0.8pt,plotstyle=line]{\mydata}

\savedata{\mydata}[
{{-5.545, -7959.39}, {-5.5, -6121.59}, {-5.45, -4282.18}, {-5.4, -2642.09}, {-5.35,
-1187.52}, {-5.3, 94.6531}, {-5.25, 1216.92}, {-5.2, 2191.16}, {-5.15,
   3028.66}, {-5.1, 3740.12}, {-5.05, 4335.69}, {-5., 4825.}, {-4.95,
  5217.14}, {-4.9, 5520.72}, {-4.85, 5743.85}, {-4.8,
  5894.19}, {-4.75, 5978.96}, {-4.7, 6004.93}, {-4.65,
  5978.48}, {-4.6, 5905.56}, {-4.55, 5791.79}, {-4.5,
  5642.37}, {-4.45, 5462.18}, {-4.4, 5255.75}, {-4.35,
  5027.31}, {-4.3, 4780.74}, {-4.25, 4519.67}, {-4.2,
  4247.42}, {-4.15, 3967.05}, {-4.1, 3681.36}, {-4.05, 3392.9}, {-4.,
  3104.}, {-3.95, 2816.76}, {-3.9, 2533.08}, {-3.85, 2254.65}, {-3.8,
  1982.97}, {-3.75, 1719.37}, {-3.7, 1465.01}, {-3.65,
  1220.89}, {-3.6, 987.858}, {-3.55, 766.62}, {-3.5, 557.758}, {-3.45,
   361.725}, {-3.4, 178.863}, {-3.35,
  9.40359}, {-3.3, -146.519}, {-3.25, -288.864}, {-3.2, -417.672},
{-3.15, -533.066}, {-3.1, -635.238}, {-3.05, -724.444}, {-3., -801.},
{-2.95, -865.275}, {-2.9, -917.685}, {-2.85, -958.688}, {-2.8,
-988.778}, {-2.75, -1008.48}, {-2.7, -1018.35}, {-2.65, -1018.97},
{-2.6, -1010.92}, {-2.55, -994.823}, {-2.5, -971.289}, {-2.45,
-940.947}, {-2.4, -904.425}, {-2.35, -862.354}, {-2.3, -815.358},
{-2.25, -764.06}, {-2.2, -709.07}, {-2.15, -650.99}, {-2.1,
-590.406}, {-2.05, -527.892}, {-2., -464.}, {-1.95, -399.267}, {-1.9,
-334.206}, {-1.85, -269.309}, {-1.8, -205.043}, {-1.75, -141.853},
{-1.7, -80.1539}, {-1.65, -20.3361}, {-1.6, 37.2384}, {-1.55,
  92.2359}, {-1.5, 144.352}, {-1.45, 193.31}, {-1.4, 238.865}, {-1.35,
   280.8}, {-1.3, 318.928}, {-1.25, 353.093}, {-1.2, 383.166}, {-1.15,
   409.049}, {-1.1, 430.673}, {-1.05, 447.994}, {-1., 461.}, {-0.95,
  469.703}, {-0.9, 474.142}, {-0.85, 474.383}, {-0.8,
  470.513}, {-0.75, 462.646}, {-0.7, 450.917}, {-0.65,
  435.482}, {-0.6, 416.518}, {-0.55, 394.221}, {-0.5,
  368.805}, {-0.45, 340.499}, {-0.4, 309.548}, {-0.35,
  276.212}, {-0.3, 240.762}, {-0.25, 203.478}, {-0.2,
  164.653}, {-0.15, 124.586}, {-0.1, 83.5804}, {-0.05, 41.9475}, {0.,
  0.}, {0.05, -41.9475}, {0.1, -83.5804}, {0.15, -124.586}, {0.2,
-164.653}, {0.25, -203.478}, {0.3, -240.762}, {0.35, -276.212}, {0.4,
-309.548}, {0.45, -340.499}, {0.5, -368.805}, {0.55, -394.221}, {0.6,
-416.518}, {0.65, -435.482}, {0.7, -450.917}, {0.75, -462.646}, {0.8,
-470.513}, {0.85, -474.383}, {0.9, -474.142}, {0.95, -469.703}, {1.,
-461.}, {1.05, -447.994}, {1.1, -430.673}, {1.15, -409.049}, {1.2,
-383.166}, {1.25, -353.093}, {1.3, -318.928}, {1.35, -280.8}, {1.4,
-238.865}, {1.45, -193.31}, {1.5, -144.352}, {1.55, -92.2359}, {1.6,
-37.2384}, {1.65, 20.3361}, {1.7, 80.1539}, {1.75, 141.853}, {1.8,
  205.043}, {1.85, 269.309}, {1.9, 334.206}, {1.95, 399.267}, {2.,
  464.}, {2.05, 527.892}, {2.1, 590.406}, {2.15, 650.99}, {2.2,
  709.07}, {2.25, 764.06}, {2.3, 815.358}, {2.35, 862.354}, {2.4,
  904.425}, {2.45, 940.947}, {2.5, 971.289}, {2.55, 994.823}, {2.6,
  1010.92}, {2.65, 1018.97}, {2.7, 1018.35}, {2.75, 1008.48}, {2.8,
  988.778}, {2.85, 958.688}, {2.9, 917.685}, {2.95, 865.275}, {3.,
  801.}, {3.05, 724.444}, {3.1, 635.238}, {3.15, 533.066}, {3.2,
  417.672}, {3.25, 288.864}, {3.3,
  146.519}, {3.35, -9.40359}, {3.4, -178.863}, {3.45, -361.725},
{3.5, -557.758}, {3.55, -766.62}, {3.6, -987.858}, {3.65, -1220.89},
{3.7, -1465.01}, {3.75, -1719.37}, {3.8, -1982.97}, {3.85, -2254.65},
{3.9, -2533.08}, {3.95, -2816.76}, {4., -3104.}, {4.05, -3392.9},
{4.1, -3681.36}, {4.15, -3967.05}, {4.2, -4247.42}, {4.25, -4519.67},
{4.3, -4780.74}, {4.35, -5027.31}, {4.4, -5255.75}, {4.45, -5462.18},
{4.5, -5642.37}, {4.55, -5791.79}, {4.6, -5905.56}, {4.65, -5978.48},
{4.7, -6004.93}, {4.75, -5978.96}, {4.8, -5894.19}, {4.85, -5743.85},
{4.9, -5520.72}, {4.95, -5217.14}, {5., -4825.}, {5.05, -4335.69},
{5.1, -3740.12}, {5.15, -3028.66}, {5.2, -2191.16}, {5.25, -1216.92},
{5.3, -94.6531}, {5.35, 1187.52}, {5.4, 2642.09}, {5.45,
  4282.18}, {5.5, 6121.59}, {5.545, 7959.39}}
  ]
\dataplot[linecolor=black,linewidth=0.8pt,plotstyle=line]{\mydata}

\end{pspicture}
\caption{Renormalized Jensen polynomials for $b_2(n)$ approaching $H_7(X/2)$ \label{b2fig}}
\end{center}
\end{figure}


Figure \ref{b2fig} illustrates Theorem \ref{kreg} for $k=2$ and the sequence of partitions into distinct parts. The graphs of the renormalized Jensen polynomials on the left of \e{hjak} are displayed for $d=7$ and $n=10^7$ with $A(n)$ and $\delta(n)$ given by \e{pgb}. They appear close to the graph of $H_d(X/2)$, which is in the middle.

In sections \ref{other} and \ref{last} we  give further general families of log-polynomial sequences and it is seen that the Hermite-Jensen property  is a very general phenomenon. See also \cite{Far} where this `Hermite universality' is discussed for the Taylor coefficients of entire functions.

\section{Some  polynomial families} \label{polys}

The {\em De Moivre polynomials} $\dm_{n,k}(a_1, a_2, a_3, \dots)$ appear when taking powers of power series:
\begin{equation} \label{bell}
    \left( a_1 x +a_2 x^2+ a_3 x^3+ \cdots \right)^k = \sum_{n=k}^\infty \dm_{n,k}(a_1, a_2, a_3, \dots) x^n \qquad \quad (k \in \Z_{\gqs 0}).
\end{equation}
Expanding the left side of \e{bell} with multinomial coefficients then shows
\begin{equation} \label{bell2}
  \dm_{n,k}(a_1, a_2, a_3, \dots) = \sum_{\substack{1j_1+2 j_2+ \dots +mj_m= n \\ j_1+ j_2+ \dots +j_m= k}}
 \binom{k}{j_1 , j_2 ,  \dots , j_m} a_1^{j_1} a_2^{j_2}  \cdots a_m^{j_m}
\end{equation}
where $m=n-k+1$ and the sum   is over all possible $j_1$, $j_2$,  \dots , $j_m \in \Z_{\gqs 0}$.
Therefore $\dm_{n,k}(a_1, a_2, a_3, \dots)$ is a polynomial in $a_1, a_2, \dots, a_{n-k+1}$ of homogeneous degree $k$ with positive integer coefficients.   They were introduced by De Moivre in 1697 and are also known in the literature as a type of Bell polynomial. See \cite[Sect. 3.3]{Comtet}, \cite{odm} for more information. We will only need the following easily established properties:
\begin{align}\label{pw}
  \dm_{n,k}(c a_1, c^2 a_2, c^3 a_3, \dots) & = c^n \dm_{n,k}(a_1, a_2, a_3, \dots),\\
\label{expl}
  \dm_{n,k}(x,y,0,0,\dots) & = \binom{k}{n-k} x^{2k-n}y^{n-k} \qquad (n\gqs k).
\end{align}

If we set
\begin{equation} \label{hhh}
   h_r(y) :=
                                 \begin{cases}
                                   0, & \text{if $r$ is odd;} \\
                                   y^{r/2}/(r/2)!, & \text{if $r$ is even,}
                                 \end{cases}
\end{equation}
then it is clearly true that
\begin{equation} \label{genfu}
  e^{2X t+y t^2}= e^{2X t} \cdot e^{y t^2} = \sum_{d=0}^\infty \left[ d! \sum_{r=0}^d h_{d-r}(y) \frac{(2X)^r}{r!}\right] \frac{t^d}{d!}.
\end{equation}
With $y=-1$ in \e{genfu} we have the generating function of the {\em Hermite polynomials} $H_d(X)$:
\begin{equation}\label{her}
  H_d(X)= d! \sum_{r=0}^d h_{d-r}(-1) \frac{(2X)^r}{r!}.
\end{equation}
We will  need the $y=1$ case as well. It can  be related to $H_d(X)$ by replacing $X$ by $iX$ and $t$ by $t/i$ in \e{genfu}:
\begin{equation}\label{her2}
   d! \sum_{r=0}^d h_{d-r}(1) \frac{(2X)^r}{r!} = i^{-d} H_d(i X).
\end{equation}
Expanding the left of \e{genfu} we may also write
\begin{equation} \label{her3}
  d! \sum_{r=0}^d h_{d-r}(y) \frac{(2X)^r}{r!} = \sum_{k=0}^d \frac{d!}{k!} \dm_{d,k}(2X,y,0,0,\dots).
\end{equation}
Hence \e{expl} and \e{her3} imply the simple formulas \e{mv1} and \e{mv2}.

 The  {\em generalized  Laguerre polynomials} will also be required in section \ref{last}:
\begin{equation}\label{lagu}
  L_d^{(r)}(X):= \frac{\G(d+r+1)}{d!} \sum_{j=0}^d \binom{d}{j} \frac{(-X)^j}{\G(j+r+1)}.
\end{equation}
The Hermite polynomials may be expressed in terms of the Laguerre polynomials in \e{lagu} with $r=\pm 1/2$; see \cite[p. 284]{AAR}.

 Proposition \ref{logp} is proved next, giving some first properties of log-polynomial sequences. We need a lemma.

\begin{lemma} \label{delt}
Fix $m\gqs 0$. Suppose that $P_n(j) = o(\Phi(n))$ as $n \to \infty$ when $j=0,1,2,\dots,m$ for the polynomial
$$
P_n(x):= \phi_0(n)  +  \phi_1(n) x +  \phi_2(n) x^2 + \cdots  +\phi_m(n) x^m.
$$
 Then $\phi_k(n)= o(\Phi(n))$ for $k=0,1,2,\dots,m$.
\end{lemma}
\begin{proof}
Use induction on $m$ with the case $m=0$ clearly true. The forward difference operator $\Delta$ acts on functions by $\Delta f(x):= f(x+1)-f(x)$, and repeated application gives
\begin{equation} \label{fres}
  \Delta^r f(x) = \sum_{j=0}^r \binom{r}{j} (-1)^{r+j} f(x+j).
\end{equation}
Also, see \cite[Sect. 1.6]{Comtet},
\begin{equation} \label{fres2}
  \left.\Delta^r x^k \right|_{x=0} = \sum_{j=0}^r \binom{r}{j} (-1)^{r+j} j^k = r! \stirb{k}{r}
\end{equation}
where the Stirling subset numbers $\stirb{k}{r}$ count the number of ways to partition  $k$ elements into $r$ nonempty subsets with $\stirb{k}{k}=1$ and $\stirb{k}{r}=0$ for $r>k$.
Therefore
\begin{equation*}
  \sum_{j=0}^m \binom{m}{j} (-1)^{m+j} P_n(j)  = \left.\Delta^m P_n(x) \right|_{x=0}
    = \sum_{k=0}^m \phi_k(n) m! \stirb{k}{m}
 =  m! \phi_m(n).
\end{equation*}
It follows that $\phi_m(n)= o(\Phi(n))$. We can then include this highest degree term in the error and assume $P_n(x)$ is of degree $m-1$, allowing the induction to proceed.
\end{proof}

\begin{proof}[Proof of Proposition \ref{logp}] Parts (i) and (ii) are clear. For part (iii),
 assume that \e{hjk} and its associated conditions are true and also a starred version:
 there exist sequences $A^*(n)$, $\delta^*(n)$ and $g_k^*(n)$ for $k=3,4, \dots, m$  so that
\begin{equation} \label{hjkx}
    \log\left(\frac{\alpha(n+j)}{\alpha(n)}\right) = A^*(n) j +\kappa^* \cdot \delta^*(n)^2 j^2 + \sum_{k=3}^{m} g_k^*(n) j^k + o(\delta^*(n)^{m+1})
  \end{equation}
as $n \to \infty$, for $j=1,2, \dots, m+1$. Here  $\delta^*(n)>0$, $\delta^*(n) \to 0$ and $g_k^*(n)=o(\delta^*(n)^k)$ as $n \to \infty$.

 By Proposition \ref{first} we must have $\kappa^* = \kappa$. Then the difference of \e{hjk} and \e{hjkx} has
\begin{equation*}
  \left(A^*(n)-A(n)\right) j +\left(\kappa \cdot \delta^*(n)^2-\kappa \cdot \delta(n)^2\right) j^2 + \sum_{k=3}^{m} \left(g^*_k(n)-g_k(n)\right) j^k = o\left(\delta^*(n)^{m+1} + \delta(n)^{m+1}\right)
\end{equation*}
as $n \to \infty$, for $j=0,1, \dots, m+1$. Apply Lemma \ref{delt} to find
\begin{align}
  A^*(n)-A(n) & = o\left(\delta^*(n)^{m+1} + \delta(n)^{m+1}\right), \label{ut}\\
   \delta^*(n)^2- \delta(n)^2 & = o\left(\delta^*(n)^{m+1} + \delta(n)^{m+1}\right).\label{ut2}
\end{align}
We may simplify the error estimates in \e{ut} and \e{ut2} as follows.
Since $m\gqs 2$, \e{ut2} implies
\begin{equation*}
  \delta^*(n)^2-  \delta(n)^2  = o\left(\delta^*(n)^{3} + \delta(n)^{3}\right).
\end{equation*}
Hence
\begin{equation} \label{hen}
  \delta^*(n)-  \delta(n)  = o\left(\delta^*(n)^{2} + \delta(n)^{2}\right).
\end{equation}
Let $h(n):=\delta^*(n)/\delta(n)-1$. Then as $n \to \infty$, \e{hen} is equivalent to
\begin{align*}
  \frac{h(n)}{\delta^*(n)h(n) + \delta^*(n) +\delta(n)} \to 0 & \iff \frac{|\delta^*(n)h(n) + \delta^*(n) +\delta(n)|}{|h(n)|} \to \infty \\
& \iff \left|\delta^*(n) + \frac{\delta^*(n) +\delta(n)}{h(n)} \right|  \to \infty \\
& \iff  \frac{\delta^*(n) +\delta(n)}{|h(n)|}   \to \infty \\
& \iff  \frac{h(n)}{\delta^*(n) +\delta(n)}   \to 0.
\end{align*}
(We may omit the subsequence of $n$s where $h(n)=0$ in the above argument.)
It follows that $\delta^*(n) =\delta(n)(1+o(\delta^*(n) +\delta(n)))$ and, more simply, $\delta^*(n) =\delta(n)(1+o(1))$. Use this in \e{ut} and \e{ut2} to complete the proof of Proposition \ref{logp}
and we  have also shown that
\begin{equation}\label{not}
  g_k^*(n)  = g_k(n)+o\left(\delta(n)^{m+1}\right).
\end{equation}
\end{proof}

\section{Jensen limits} \label{lim}
The next result is a slight extension of \cite[Thm. 8]{GORZ} and we include its proof for completeness.

\begin{theorem} \label{pj}
Let $\alpha(n)$ be a sequence of positive real numbers and fix a positive integer $d$. Suppose there exist $R(n)$, $\delta(n)$ and $C_r(n)$ for $r=0,1, \dots, d$  so that, as $n \to \infty$,
\begin{equation} \label{sinn}
  \frac{\alpha(n+j)}{\alpha(n) R(n)^j} =  C_0(n) +\sum_{r=1}^d C_r(n) \delta(n)^r j^r +o(\delta(n)^d)
\end{equation}
 for $j=1,2, \dots, d$. Suppose also that  $\delta(n)>0$, $\delta(n) \to 0$ and $C_r(n) \to c_r$ as $n \to \infty$. Then as $n\to \infty$,
\begin{align} \label{jnxx}
   \frac{\delta(n)^{-d}}{\alpha(n)} J^{d,n}_\alpha \left( \frac{\delta(n) X -1}{R(n)}\right) & \to d! \sum_{k=0}^d (-1)^{d-k} c_{d-k} \frac{X^k}{k!},\\
 \frac{(R(n) \delta(n))^{-d}}{\alpha(n)} \ja^{d,n}_\alpha \Bigl( R(n)(\delta(n) X -1)\Bigr) & \to d! \sum_{k=0}^d  c_{d-k} \frac{X^k}{k!}. \label{jnxx2}
\end{align}
\end{theorem}

 We necessarily have 
 $c_0=1$ and the polynomials on the right of \e{jnxx} and \e{jnxx2} are of degree $d$ and monic.
The proof of Theorem \ref{pj} relies on the following key combinatorial lemma.

\begin{lemma} \label{finn}
For integers $d,k,r$ with $d\gqs k$ and $r\gqs 0$, set
\begin{equation*}
  \sigma(d,k;r):=  \sum_{j=k}^d (-1)^{j-k} \binom{d-k}{j-k} j^r.
\end{equation*}
Then $\sigma(d,k;r)=0$ if $r<d-k$ and $\sigma(d,k;d-k)=(-1)^{d-k}(d-k)!$.
\end{lemma}
\begin{proof}
Employing generating functions 
we have
\begin{align*}
  \sum_{r=0}^\infty \sigma(d,k;r) \frac{z^r}{r!} & = \sum_{j=k}^d (-1)^{j-k} \binom{d-k}{j-k} e^{jz} \\
  & = e^{k z}\sum_{j=k}^d  \binom{d-k}{j-k} (-e^z)^{j-k} = e^{k z}(1-e^z)^{d-k}
\end{align*}
and the result follows as this last expression is $$
\left(1+k z+O(z^2)\right)\left(1-1-z+O(z^2)\right)^{d-k}
=(-1)^{d-k}z^{d-k}+ O(z^{d-k+1}).
$$

For a second proof, as in \cite[Sect. 2]{GORZ}, $\sigma(d,k;r)$ may be recognized with \e{fres} as an iterated forward difference of the monomial $x^r$ with
\begin{equation} \label{pop}
  \sum_{j=k}^d (-1)^{j-k} \binom{d-k}{j-k} j^r  = \left. (-1)^{d-k} \Delta^{d-k} x^r \right|_{x=k}.
\end{equation}
Since every application of $\Delta$ reduces the degree by $1$, \e{pop} must equal $0$ for $r<d-k $. For   $r=d-k$ the degree of $ \Delta^{r} x^r$ is zero and \e{pop} equals $(-1)^{r}r!$ by \e{fres2}.
\end{proof}

\begin{proof}[Proof of Theorem \ref{pj}]
Write
\begin{align*}
  \frac{\delta(n)^{-d}}{\alpha(n)} J^{d,n}_\alpha \left( \frac{\delta(n) X -1}{R(n)}\right)
& = \frac{\delta(n)^{-d}}{\alpha(n)} \sum_{j=0}^d \binom{d}{j} \alpha(n+j) \left( \frac{\delta(n) X -1}{R(n)}\right)^j \\
 & =\sum_{k=0}^d \binom{d}{k} \delta(n)^{k-d} X^k \sum_{j=k}^d (-1)^{j-k}\binom{d-k}{j-k} \frac{\alpha(n+j)}{\alpha(n) R(n)^j}.
\end{align*}
Insert \e{sinn} and interchange the order of summation to find that the coefficient of $X^k$ equals
\begin{equation} \label{pinn}
  \binom{d}{k} \sum_{r=0}^d C_r(n) \delta(n)^{k-d+r} \sum_{j=k}^d (-1)^{j-k} \binom{d-k}{j-k} j^r + o\left(\delta(n)^k\right).
\end{equation}
By Lemma \ref{finn}, \e{pinn} is
\begin{equation} \label{rahm}
  \frac{d!}{k!} (-1)^{d-k} C_{d-k}(n) + \binom{d}{k} \sum_{r=d-k+1}^d C_r(n) \delta(n)^{k-d+r} \sigma(d,k;r) + o\left(\delta(n)^k\right)
\end{equation}
and \e{jnxx} follows. The limit \e{jnxx2} is shown similarly with $(-1)^{d-k}\sigma(d-k,0;r)$ appearing in place of $\sigma(d,k;r)$.
\end{proof}


\begin{lemma} \label{finnc}
Fix $j\in \R$ and suppose $\delta(n) \to 0$ as $n \to \infty$ and $\rho(n)=o(\delta(n))$. Then for every positive integer $m$  we can write
\begin{equation*}
  (1+\rho(n))^j = 1 + \sum_{\ell=1}^{m-1} \rho_\ell(n) j^\ell + o\left(\delta(n)^m \right) \qquad \text{with} \qquad \rho_\ell(n)=o\left(\delta(n)^\ell \right).
\end{equation*}
\end{lemma}
\begin{proof}
We have the useful elementary inequality
\begin{equation}\label{uin}
  |\log(1+z)| \lqs 2|z| \qquad \text{for} \qquad z\in \C, \ |z|\lqs 3/4.
\end{equation}
Therefore $j \log(1+\rho(n))$ is small for large $n$ and so
\begin{equation*}
  (1+\rho(n))^j = e^{j \log(1+\rho(n))} = 1 + \sum_{\ell=1}^{m-1} \frac 1{\ell!}\log^\ell(1+\rho(n)) \cdot j^\ell + O\bigl( |\log(1+\rho(n))|^m j^m\bigr).
\end{equation*}
The lemma follows as $\log(1+\rho(n)) = O(|\rho(n)|) = o(\delta(n))$ by \e{uin}.
\end{proof}

Lemma \ref{finnc} is used next to show there is some flexibility in how we choose $R(n)$ and $\delta(n)$ in \e{sinn}.

\begin{prop} \label{pjxx}
Let $\alpha(n)$ be a sequence of positive real numbers and fix a positive integer $d$. Suppose there exist $R(n)$, $\delta(n)$ and $C_r(n)$ for $r=0,1, \dots, d$ satisfying the conditions of Theorem \ref{pj}. This means that, as $n \to \infty$,
\begin{equation} \label{sinnxx}
  \frac{\alpha(n+j)}{\alpha(n) R(n)^j} = \sum_{r=0}^d C_r(n) \delta(n)^r j^r +o(\delta(n)^d)
\end{equation}
 for $j=1,2, \dots, d$. Also  $\delta(n)>0$, $\delta(n) \to 0$ and $C_r(n) \to c_r$ as $n \to \infty$. If $R^*(n)$ and $\delta^*(n)$ satisfy
\begin{equation} \label{ophaxx}
   R^*(n) = R(n)(1+o(\delta(n))), \qquad  \delta^*(n) = \delta(n)(1+o(1)) \qquad \text{as \quad $n\to \infty$}
\end{equation}
then there exist $C^*_r(n)$ so that \e{sinnxx}  remains true with $R(n)$, $C_r(n)$ and $\delta(n)$ replaced by  their starred versions. Precisely,  as $n \to \infty$,
\begin{equation} \label{sinnxx2}
  \frac{\alpha(n+j)}{\alpha(n) R^*(n)^j} = \sum_{r=0}^d C^*_r(n) \delta^*(n)^r j^r +o(\delta^*(n)^d)
\end{equation}
 for $j=1,2, \dots, d$, where $C^*_r(n) \to c_r$.
\end{prop}
\begin{proof}
For clarity, write $R^*(n) = R(n)(1+\rho(n))$  and $\delta^*(n) = \delta(n)(1+\phi(n))$ where $\rho(n)=o(\delta(n))$ and $\phi(n)=o(1)$. Then \e{sinnxx} implies
\begin{equation} \label{jul}
  \frac{\alpha(n+j)}{\alpha(n) R^*(n)^j} = (1+\rho(n))^{-j}\left(\sum_{r=0}^d C_r(n) (1+\phi(n))^{-r} \delta^*(n)^r j^r +o(\delta^*(n)^d)\right).
\end{equation}
With Lemma \ref{finnc}, write
$$
(1+\rho(n))^{-j} = \sum_{\ell=0}^{d} \rho_\ell(n) j^\ell + o\left(\delta^*(n)^{d+1} \right)
\qquad \text{for} \qquad \rho_\ell(n) := \frac {(-1)^\ell}{\ell!}\log^\ell(1+\rho(n))
$$
and $\rho_\ell(n) = o(\delta^*(n)^\ell)$. Hence
\begin{multline*}
  \frac{\alpha(n+j)}{\alpha(n) R^*(n)^j} = \sum_{u=0}^d \delta^*(n)^u j^u \sum_{\ell =0}^u C_{u-\ell}(n) (1+\phi(n))^{\ell-u} \frac{\rho_\ell(n)}{\delta^*(n)^\ell} \\
  + \sum_{u=d+1}^{2d} \delta^*(n)^u j^u \sum_{\ell =u-d}^d C_{u-\ell}(n) (1+\phi(n))^{\ell-u} \frac{\rho_\ell(n)}{\delta^*(n)^\ell} + o(\delta^*(n)^d).
\end{multline*}
The sum with $u$ between $d+1$ and $2d$ is $o(\delta^*(n)^d)$. Consequently we obtain \e{sinnxx2} for
\begin{equation*}
  C^*_r(n) := C_{r}(n) (1+\phi(n))^{-r} + \sum_{\ell =1}^r C_{r-\ell}(n) (1+\phi(n))^{\ell-r} \frac{\rho_\ell(n)}{\delta^*(n)^\ell}.
\end{equation*}
The bounds for $\phi(n)$ and $\rho_\ell(n)$ now ensure that $C^*_r(n) \to c_r$ as $n \to \infty$.
\end{proof}

\section{Main theorems} \label{main}


\begin{proof}[Proof of Theorem \ref{vog}]
We have
\begin{equation} \label{ac3}
    \log\left(\frac{\alpha(n+j)}{\alpha(n)}\right) = A(n) j +\kappa \cdot \delta(n)^2 j^2 + \sum_{k=3}^{m} g_k(n) j^k + o\left(\delta(n)^{m+1}\right)
  \end{equation}
 for $j=1,2, \dots, m+1$. Also  $\delta(n)>0$, $\delta(n) \to 0$ and $g_k(n)=o(\delta(n)^k)$ as $n \to \infty$.
Let $Y(n,j):=\kappa \cdot  \delta(n)^2 j^2 + \sum_{k=3}^{m} g_k(n) j^k$. Then $Y(n,j)\ll \delta(n)^2$ and
\e{ac3} implies
\begin{align}
   \frac{\alpha(n+j)}{\alpha(n) \exp(A(n))^j} & = \left(\sum_{u=0}^t \frac 1{u!} Y(n,j)^u +O\left(|Y(n,j)|^{t+1} \right) \right)\left(1+ o\left(\delta(n)^{m+1}\right)\right) \notag\\
& = \sum_{u=0}^t \frac 1{u!} Y(n,j)^u + O\left(\delta(n)^{2t+2}\right)+o\left(\delta(n)^{m+1}\right). \label{hail}
\end{align}
Now, recalling \e{bell} and \e{pw},
\begin{align*}
  \sum_{u=0}^{t} \frac 1{u!} Y(n,j)^u &
= \sum_{u=0}^{t} \frac 1{u!} \left(\kappa \cdot  \delta(n)^2 j^2 + \sum_{k=3}^{m} g_k(n) j^k\right)^u \\
   & = \sum_{u=0}^{t} \frac 1{u!} \sum_{r=u}^{m u} \dm_{r,u}\left(0,\kappa \cdot \delta(n)^2, g_3(n), \dots, g_{m}(n),0,\dots \right) j^r\\
& = \sum_{u=0}^{t} \frac 1{u!} \sum_{r=u}^{m u} \dm_{r,u}\left(0,\kappa, \frac{g_3(n)}{\delta(n)^3}, \dots, \frac{g_{m}(n)}{\delta(n)^{m}},0,\dots \right) \delta(n)^r j^r\\
& = \sum_{r=0}^{m t} C_r(n) \delta(n)^r j^r
\end{align*}
for
$$
C_r(n) = \sum_{u=0}^{t} \frac 1{u!} \dm_{r,u}\left(0,\kappa, \frac{g_3(n)}{\delta(n)^3}, \dots, \frac{g_{m}(n)}{\delta(n)^{m}},0,\dots \right).
$$
To make the error in \e{hail} of size $o(\delta(n)^d)$ we choose $d$ with $1\lqs d\lqs m+1$ and then $t\gqs 1$ so that $2t\gqs d$.
Our assumptions for $g_k(n)$ mean that $C_r(n)\ll 1$ and so
\begin{equation}\label{alm}
\frac{\alpha(n+j)}{\alpha(n) \exp(A(n))^j} = C_0(n)+\sum_{r=1}^{d} C_r(n) \delta(n)^r j^r + o\left(\delta(n)^d\right).
\end{equation}
Also
$$
\lim_{n \to \infty} C_r(n) = \sum_{u=0}^{t} \frac 1{u!} \dm_{r,u}\left(0,\kappa, 0,0, \dots, \right),
$$
where $\dm_{r,u}\left(0,\kappa, 0,0, \dots, \right)$ equals $\kappa^u$ if $r=2u$ and  otherwise equals $0$.
Therefore, for $r\lqs d \lqs 2t$,
\begin{equation} \label{crnt}
  \lim_{n\to \infty}  C_r(n) = c_r =
                                 \begin{cases}
                                   0, & \text{if $r$ is odd;} \\
                                   \kappa^{r/2}/(r/2)!, & \text{if $r$ is even.}
                                 \end{cases}
\end{equation}
Theorem \ref{pj} now applies to \e{alm} with $R(n)=\exp(A(n))$. The  polynomials on the right of \e{jnxx} and \e{jnxx2} are recognized by \e{her} and \e{her2} as the desired ones on the right of \e{hjak}. This finishes the proof of Theorem \ref{vog}.
\end{proof}

\begin{proof}[Proof of Theorem \ref{vog2}]
Using the proof of Theorem \ref{vog} we obtain \e{alm} satisfying the conditions of Theorem \ref{pj} with $R(n)=\exp(A(n))$. For $A^*(n)$ and $\delta^*(n)$ satisfying \e{tre}, set $R^*(n)=\exp(A^*(n))$. Then $R^*(n) = R(n)(1+o(\delta(n)))$ and   Proposition \ref{pjxx} implies that \e{alm} is also valid with $A(n)$, $\delta(n)$ and $C_r(n)$ replaced by $A^*(n)$, $\delta^*(n)$ and $C^*_r(n)$, with $C^*_r(n)$ giving the same limit \e{crnt}. Applying Theorem \ref{pj} now  completes the proof.
\end{proof}

The next result lets us prove the properties mentioned in Theorem \ref{redct} for the zeros of shifted Jensen polynomials associated to Hermite-Jensen sequences.

\begin{theorem} \label{dct2} \cite{US}
 Let $Q_n(X)$ be a sequence of monic polynomials of the same degree in $\C[x]$ so that
\begin{equation} \label{qlim}
  \lim_{n\to \infty} Q_n(X) = Q(X).
\end{equation}
Let $r$ be any root  of $Q(X)$ and suppose it has  multiplicity $m$. Then for every $\epsilon>0$ we can find an $N$ so that $Q_n(X)$ has $m$ roots  inside the complex ball of radius $\epsilon$ centered at $r$  for all $n \gqs N$.
\end{theorem}

\begin{proof}[Proof of Theorem \ref{redct}]
Part (i) follows from Proposition \ref{first}. For part (ii), it is clearly enough to prove it for $J_\alpha^{d,n}(X)$. Theorem \ref{vog} implies the sequence is Hermite-Jensen of degree $d$. Hence there exist sequences of positive reals $A_n, B_n, C_n$ so that
\begin{equation} \label{1990}
 A_n J_\alpha^{d,n}(B_n X - C_n) \to Q(X):= \begin{cases}
H_d(X/2) & \text{ if $\kappa=-1$,} \\
i^{-d} H_d(i X/2) & \text{ if $\kappa=1$.}
\end{cases}
\end{equation}
 The coefficient of $X^d$ on the left of \e{1990} is $\alpha(n+d)A_n B_n^d$ which is nonzero and tends to $1$. Divide by this to get the  polynomials
\begin{equation*}
  Q_n(X):=  J_\alpha^{d,n}(B_n X - C_n)/(\alpha(n+d) B_n^d).
\end{equation*}
Then $Q_n(X)$ is monic of degree $d$ with real coefficients and tends to $Q(X)$ as $n \to \infty$.

The Hermite polynomials $H_d(X)$ have distinct real roots with $0$ as a root exactly when $d$ is odd.
Choose $\epsilon$ in Theorem \ref{dct2} small enough so that the balls of this radius around each root of $Q(x)$ don't overlap. It follows  that there exists an $N$ so that the roots of $Q_n(X)$ are distinct for $n \gqs N$. Hence $J_\alpha^{d,n}(X)$ must also have distinct roots for $n$ large enough.

If $\kappa=-1$ then the roots of $Q(X)$ are real. The roots of $Q_n(X)$ for $n \gqs N$ must also be real as otherwise two conjugate roots would lie within the same ball. Therefore the roots of $J_\alpha^{d,n}(X)$ must also be real in this case.
If $\kappa=1$ then the roots of $Q(X)$ are all imaginary (and not real) except for the root $0$ when $d$ is odd. Choose a smaller $\epsilon$ if necessary so that the balls around the non-real roots do not intersect the real line. If $d$ is odd then the single root of $Q_n(X)$ in the ball about $0$ must be real for $n \gqs N$. All other roots of $Q_n(X)$  must  be not real. Therefore the roots of $J_\alpha^{d,n}(X)$ must also be non-real in this case with the exception of one real root when $d$ is odd.
\end{proof}

\section{Partitions} \label{p-proof}

The general properties of asymptotic expansions such as
$$
f(x)=a_0+\frac{a_1}{x}+\frac{a_2}{x^2}+\frac{a_3}{x^3}+O\left( \frac1{x^4}\right) \qquad \text{when} \quad x \to \infty,
$$
are explained in \cite[pp. 16 - 22]{Ol}. These expansions are unique to each order and may be added, multiplied and divided. It is also valid to integrate the expansion of $f(x)$ to  obtain the expansion of the integral of $f(x)$. However, differentiating asymptotic expansions is not always valid in the same way. It may be justified in the following situation, as described in \cite[p. 21]{Ol}, by employing integration and uniqueness. We give the details here for the case we require.

\begin{lemma} \label{diff}
Let $f(x)$ be a differentiable function on some positive interval $[c,\infty)$  with $f'(x)$  continuous. Suppose $f$ and $f'$ have the asymptotic expansions
\begin{equation*}
  f(x)  = \sum_{j=1}^R \frac{a_j}{x^{j/2}} + O\left( \frac{1}{x^{(R+1)/2}}\right),\qquad
f'(x)  = \sum_{j=0}^S \frac{b_j}{x^{j/2}} + O\left( \frac{1}{x^{(S+1)/2}}\right)
\end{equation*}
as $x \to \infty$.
Then  $b_0=b_1=b_2=0$ and $b_{j+2}=-(j/2)a_j$ for $3\lqs j \lqs \min(R,S-2)$. In other words, we may differentiate the asymptotic expansion of $f(x)$ to get the asymptotic expansion of $f'(x)$ if $f'(x)$ is continuous and  known to have an asymptotic expansion.
\end{lemma}
\begin{proof}
The identity
$$
f(x)=f(c)+\int_c^x f'(t)\, dt
$$
shows that we must have $b_0=b_1=b_2=0$ as otherwise $f(x)$ becomes unbounded as $x\to \infty$.
Set $$
E_S(x):= f'(x)  - \sum_{j=3}^S \frac{b_j}{x^{j/2}}.
$$
Then
$$
  f(x)  = \int_\infty^x f'(t)\, dt
     = \int_\infty^x \sum_{j=3}^S \frac{b_j}{t^{j/2}}\, dt + \int_\infty^x E_S(t) \, dt
  = \sum_{j=1}^{S-2} \frac{c_j}{x^{j/2}} + O\left( \frac{1}{x^{(S-1)/2}}\right)
$$
for $b_{j+2}=-(j/2)c_j$. Hence, for $M=\min(R,S-2)$,
\begin{equation}\label{tmin}
0=f(x)-f(x) = \sum_{j=1}^{M} \frac{a_j-c_j}{x^{j/2}} + O\left( \frac{1}{x^{(T+1)/2}}\right).
\end{equation}
Multiply both sides of \e{tmin} by $x^{1/2}$ and take the limit as $x\to \infty$ to see that $a_1=c_1$. Repeat this procedure to show that $a_j=c_j$ for $j\lqs M$. This completes the proof.
\end{proof}

We will also use the next  result following from basic complex variables; see \cite[pp. 125-126]{Al}.
\begin{lemma} \label{tay}
Suppose $F(z)$ is holomorphic for $\Re(z)>T_0$ and for $T>2T_0$ we have the bound
\begin{equation*}
  |F(T(1+\lambda))| \lqs \beta(T) \qquad \text{for} \qquad \lambda \in \C, \ |\lambda|\lqs 1/2.
\end{equation*}
Then for all $T>2T_0$ and $|\lambda|\lqs 1/4$, say, we have
\begin{equation*}
  F(T(1+\lambda)) = \sum_{r=0}^{m} \frac{T^r}{r!} F^{(r)}(T) \cdot \lambda^r +O\bigl(\beta(T) \cdot |\lambda|^{m+1}\bigr),
\end{equation*}
where the implied constant depends only on $m$.
\end{lemma}



The method we use next to prove  that the sequence $p(n)$ is Hermite-Jensen will also work for  overpartitions, $k$-regular partitions and many other examples.

\begin{proof}[Proof of Theorem \ref{partit}]
Put $\g:=2\pi^2/3$ and $n':=n-1/24$ for convenience, and also set
$$
T:= \g n', \qquad \lambda:=j/n', \qquad F(z):=-\log z+\log\left( 1-\frac 1{\sqrt{z}}\right)+\sqrt{z}.
$$
Then
\e{hard2} implies
\begin{equation*}
  \log(p(n+j))=\log\left(\frac{\g}{4\sqrt{3}}\right) + F(T(1+\lambda))+O(e^{- \sqrt{\g n}/2}).
\end{equation*}
Let $m\gqs 2$ be a fixed integer. By Lemma \ref{tay} we have the expansion
$$
 F(T(1+\lambda)) = \sum_{r=0}^{m} \frac{T^r}{r!} F^{(r)}(T) \cdot \lambda^r +O\left( n^{1/2}|\lambda|^{m+1}\right)
$$
and hence, for $j$ in the range we want, $1\lqs j \lqs m+1$,
\begin{equation} \label{wimb}
  \log\left(\frac{p(n+j)}{p(n)}\right) = \sum_{r=1}^{m} \frac{\g^{r}}{r!} F^{(r)}(T) \cdot j^r +O\left( \frac{1}{n^{m+1/2}}\right)
\end{equation}
as $n \to \infty$ for an implied constant depending only on $m$. A comparison of \e{wimb} with \e{hjk} now gives explicit expressions for $A(n)$, $\kappa \cdot\delta(n)^2$ and $g_k(n)$ in terms of derivatives of $F(z)$.

We wish to apply Lemma \ref{diff} to these derivatives to obtain their asymptotic expansions. To do this write
\begin{equation} \label{toca}
  F'(z)=\frac{1}{2 z^{1/2}}-\frac{1}{z}+\frac{1}{2 z^{3/2}} G(\sqrt{z}) \qquad \text{for} \qquad G(z):=\frac{1}{1-1/z}.
\end{equation}
Then $G(z)$ has the nice properties $G'(z)= - G(z)^2/z^2$ and
\begin{equation} \label{gex}
  G(z)=\sum_{j=0}^m \frac{1}{z^j}+O\left( \frac 1{|z|^{m+1}}\right) \qquad (m \in \Z_{\gqs 0}, |z|>1).
\end{equation}
Therefore, when $x >1$ and working inductively, $F^{(r)}(x)$ may be expressed as a linear combination of terms of the form $G(\sqrt{x})^a/x^{b/2}$ where $a,b \in \Z_{\gqs 0}$. It then follows from \e{gex} that $F^{(r)}(x)$
has an asymptotic expansion  of the form
\begin{equation*}
  F^{(r)}(x) = \sum_{j=0}^S \frac{c_j}{x^{j/2}} + O\left( \frac{1}{x^{(S+1)/2}}\right),
\end{equation*}
for any $S$, as $x \to \infty$. The first case, by \e{toca}, has
\begin{equation} \label{toca2}
  F'(x)=\frac{1}{2 x^{1/2}}-\frac{1}{x^{2/2}}+\frac{1}{2 x^{3/2}} +\frac{1}{2 x^{4/2}}+ \cdots +\frac{1}{2 x^{S/2}}+ O\left( \frac{1}{x^{(S+1)/2}}\right).
\end{equation}
It is also clear that $F^{(r)}(x)$ is continuous
  for $x >1$.  Lemma \ref{diff} now tells us that successive derivatives of both sides of \e{toca2} agree.  From \e{wimb} we have
\begin{align}
  A(n) & = \frac{\g}{1!} F'(T) = \frac{\g}{2\sqrt{T}}+O\left( \frac 1{T}\right), \label{sf}\\
  -\delta(n)^2 & = \frac{\g^{2}}{2!} F''(T) = -\frac{\g^{2}}{8 T^{3/2}}+O\left( \frac 1{T^2}\right),\label{sf2}\\
g_k(n) & = \frac{\g^{k}}{k!} F^{(k)}(T) = O\left( \frac 1{n^{k-1/2}}\right) \qquad (k\gqs 3). \label{sf3}
\end{align}
Then in terms of $n$,
\begin{gather}
  A(n) = \frac{\sqrt{\g}}{2\sqrt{n}}\left(1-\frac{1}{24n}\right)^{-1/2}+O\left( \frac 1{n}\right)
=  \frac{\sqrt{\g}}{2\sqrt{n}}+O\left( \frac 1{n}\right), \label{sf4}
\\
 \delta(n) 
=\frac{\g^{1/4}}{8^{1/2}n^{3/4}}\left(1-\frac{1}{24n}\right)^{-3/4}\left(1+O\left( \frac 1{n^{1/2}}\right)\right)^{1/2}
=\frac{\g^{1/4}}{8^{1/2}n^{3/4}}\left(1+O\left( \frac 1{n^{1/2}}\right)\right). \label{sf5}
\end{gather}
With \e{wimb}, \e{sf} -- \e{sf5} we have shown that the sequence $p(n)$ is log-polynomial of
degree $m$ with data $\{A(n)$, $-1$, $\delta(n)\}$ where $A(n)$ and $\delta(n)$ are given exactly in \e{sf}, \e{sf2}, simplifying to \e{part-ad}. Hence, by Theorem \ref{vog}, $p(n)$ is also Hermite-Jensen for all positive integer degrees and the same data. Take $A^*(n)$, $\delta^*(n)$ to be the main terms on the right of \e{sf4}, \e{sf5}. They satisfy the conditions of Theorem \ref{vog2} and so $p(n)$ is also Hermite-Jensen for the simpler data $\{A^*(n)$, $-1$, $\delta^*(n)\}$ in \e{part-ads}, as we wanted to show.
\end{proof}

An alternative approach supplies a more explicit version of \e{wimb}. Write
\begin{equation*}
  F(T(1+\lambda))-F(T)=-\log(1+\lambda)+\sqrt{T}\left(\sqrt{1+\lambda}- 1\right)+\log\left(
1-\frac{1}{\sqrt{T}-1}\left[(1+\lambda)^{-1/2}-1\right]
\right)
\end{equation*}
and expanding this as a series in $\lambda$ using \e{bell} produces
\begin{equation*}
  F(T(1+\lambda))-F(T)= \sum_{r=1}^\infty \rho_r \left(\sqrt{T} \right) \lambda^r
\end{equation*}
for
\begin{equation}\label{beet}
\rho_r(x) :=  \frac{(-1)^r}{r} + \binom{1/2}{r} x -\sum_{u=1}^r \frac{1}{u (x-1)^u} \dm_{r,u}\left(\binom{-1/2}{1}, \binom{-1/2}{2}, \dots \right).
\end{equation}
In particular,
\begin{equation}\label{rh}
  \rho_1(x) = \frac{1}{2}\left(\frac{x^2}{x-1}-3\right), \qquad \rho_2(x) = \frac{1}{8}\left(\frac{-x^3}{(x-1)^2}+6\right).
\end{equation}
We have proved
\begin{prop} \label{kcc}
For $j=1,2,\dots,m+1$ we have
$$
\log\left( \frac{p(n+j)}{p(n)}\right) = \sum_{r=1}^{m} \frac{\rho_r( \sqrt{\g n'})}{(n')^r} j^r +O\left( \frac{1}{n^{m+1/2}}\right)
$$
as $n\to \infty$ with an  implied constant depending only on $m$.
\end{prop}

\section{Overpartitions and other sequences} \label{other}

To help get an understanding of the kinds of sequences that are log-polynomial,  consider
\begin{equation}\label{alex}
  \alpha(n) = n^a \exp(c  \cdot n^b) \qquad (a,b,c \in \R)
\end{equation}
with $a,c$ not both $0$. Then for $1\lqs j \lqs n/2$, say,
$$
 \log\left(\frac{\alpha(n+j)}{\alpha(n)}\right) = \sum_{k=1}^\infty \left( (-1)^{k+1} \frac{a}{k} +c \cdot  n^b \binom{b}{k} \right) \frac{j^k}{n^k}.
$$
From the $k=1$ and $k=2$ coefficients
\begin{equation} \label{load}
  A(n) = -\frac{a}{n}+b c \cdot n^{b-1}, \qquad \kappa \cdot \delta(n)^2 = -\frac{a}{2n^2}+\frac{b(b-1)c}{2n^{2-b}}.
\end{equation}
The condition $b<2$ is needed to ensure that $\delta(n) \to 0$. If $b(b-1)c=0$ or $b<0$ then $\alpha(n)$ is not log-polynomial since $g_k(n) \neq o(\delta(n)^k)$. The log-polynomial conditions are met in the remaining cases and we obtain the following:

\begin{prop} \label{lar}
Let $\alpha(n)$ be given by \e{alex}. Suppose $0<b<2$, $b\neq 1$ and $c\neq 0$. Set $\kappa = \sgn(b(b-1)c)$ and define $A(n)$ and $\delta(n)$ with \e{load}.  Then, for all degrees $m\gqs 2$, this sequence is log-polynomial with data $\{A(n)$, $\kappa$, $\delta(n)\}$.
\end{prop}

Then $\alpha(n)$ satisfying Proposition \ref{lar} is Hermite-Jensen and it follows from Theorem \ref{redct} that if $b(b-1)c<0$ then the zeros of $J_\alpha^{d,n}(X)$ are all distinct and real for $n$ sufficiently large. At the `boundary' case $b=1$ we see that
$$
J_\alpha^{d,n}(X) = n^a e^{cn}\left((1+e^c X)^d +O\left( \frac 1n \right) \right)
$$
and the roots are all close to $-1/e^{c}$ for large $n$.

For another example, the sequence $\g(n)$ coming from the Taylor coefficients of the Riemann zeta function at $s=1/2$ is studied in \cite{GORZ}. In their Theorem 1 they show that the zeros of $J_\g^{d,n}(X)$ are all real for sufficiently large $n$. By \cite[Thm. 1.4]{OS21}, the simpler sequence $\alpha(n):=n^{-n}$ gives a crude approximation to $\g(n)$.  Use the method of Theorem \ref{partit} with $F(z)=-z \log z$ to show that this $\alpha(n)$ is log-polynomial with data $\{ -1-\log n, -1, (2n)^{-1/2}\}$.

Lastly in this section we consider the sequence of overpartitions $\overline p(n)$ from \e{pgf}.

\begin{theorem} \label{opartit}
With $S:=\pi  \sqrt{n}$, set
\begin{equation*}
  A(n):=\frac{\pi^2}{2} \left(\frac{1}{S-1} - \frac{3}{S^2} \right), \qquad \delta(n):= \frac{\pi^2}{2} \left(\frac{1}{2S(S-1)^2} - \frac{3}{S^4} \right)^{1/2}.
\end{equation*}
Then for every  degree $m \gqs 2$ the overpartition sequence $\ov(n)$ is log-polynomial with data $\{A(n)$, $-1$, $\delta(n)\}$. Hence $\ov(n)$ is also Hermite-Jensen for all positive integer degrees and the same data or, more simply, $\{A^*(n)$, $-1$, $\delta^*(n)\}$ for
\begin{equation*} 
   A^*(n):=\frac{\pi}{2 n^{1/2}}, \qquad  \delta^*(n) := \frac{\sqrt{\pi}}{8^{1/2} n^{3/4}}.
\end{equation*}
\end{theorem}
\begin{proof}
The proof is the same as for Theorem \ref{partit} but with $\g=\pi^2$ and $n'=n$.
\end{proof}

We can also write a similar expansion to Proposition \ref{kcc}, recalling \e{beet}:

\begin{prop} \label{kcco}
For $j=1,2,\dots,m+1$ we have
$$
\log\left( \frac{\ov(n+j)}{\ov(n)}\right) = \sum_{r=1}^{m-1} \frac{\rho_r(\pi \sqrt{n})}{n^r} j^r +O\left( \frac{1}{n^{m+1/2}}\right)
$$
as $n\to \infty$ with an  implied constant depending only on $m$.
\end{prop}

\section{$k$-regular partitions} \label{kregpa}

We first review from \cite[Sect. 4.12]{AAR} some results we will require for the Bessel functions $I_\alpha(z)$, defined in \e{ibes}.
For $\alpha$, $z \in \C$ they satisfy
\begin{equation}\label{ib}
  I_{\alpha-1}(z)-I_{\alpha+1}(z) = 2\alpha I_{\alpha}(z)/z, \qquad 2I'_{\alpha}(z)= I_{\alpha-1}(z)+I_{\alpha+1}(z),
\end{equation}
from which we obtain
\begin{equation}\label{ib2}
  I'_{0}(z)=I_1(z), \qquad I'_{1}(z)= I_{0}(z)-I_{1}(z)/z.
\end{equation}
As $|z|\to \infty$ with $|\arg z|<\pi/2$ there is the asymptotic expansion
\begin{equation}\label{ibexp}
  I_{\alpha}(z)= \frac{e^z}{\sqrt{2\pi z}}\left(
\sum_{j=0}^{m-1} \binom{j-1/2+\alpha}{j} \binom{j-1/2-\alpha}{j} \frac{j!}{(2z)^j}
+O\left( \frac 1{|z|^m}\right)\right).
\end{equation}

\begin{proof}[Proof of Theorem \ref{kreg}]
Make these definitions:
\begin{equation*}
    n'':=n+\frac{k-1}{24}, \qquad k':=\frac{2\pi^2}{3}\left(1-\frac 1k \right), \qquad  T:= k' \cdot n'', \qquad F(z):=\log \left( \frac{1}{\sqrt{z}} I_1\left( \sqrt{z}\right)\right).
\end{equation*}
Then by \e{bkas}, for $\lambda=j/n''$,
\begin{equation*}
  \log(b_k(n+j)) = \log\left(  \frac{\pi^2}{3 \sqrt{k}}\left(1-\frac 1k \right)\right) + F(T(1+\lambda))+O\left( e^{-\varepsilon_k \sqrt{n}}\right).
\end{equation*}
For $T$ large and $\lambda\in \C$ small we have $F(T(1+\lambda)) =O(\sqrt{T})=O(\sqrt{n})$ by \e{ibexp}. Hence we have the Taylor expansion
$$
 F(T(1+\lambda)) = \sum_{r=0}^{m} \frac{T^r}{r!} F^{(r)}(T) \cdot \lambda^r +O\left( \sqrt{n} |\lambda|^{m+1}\right)
$$
and so, for $j$ with $1\lqs j \lqs m+1$,
\begin{equation} \label{bop}
  \log\left(\frac{b_k(n+j)}{b_k(n)}\right) = \sum_{r=1}^{m} \frac{(k')^r}{r!} F^{(r)}(T) \cdot j^r +
O\left( \frac 1{n^{m+1/2}}\right),
\end{equation}
 as $n \to \infty$ for an implied constant depending only on $m$.
Next write
\begin{equation*}
  F'(z) = \frac{1}{2 \sqrt{z}} G(\sqrt{z}) -\frac 1{z}
 \qquad \text{for} \qquad G(z):=\frac{I_0(z)}{I_1(z)}.
\end{equation*}
It follows from \e{ibexp} that  we have the expansion
  \begin{equation} \label{gex2}
  G(z)=\sum_{j=0}^m \frac{\tau_j}{z^j}+O\left( \frac 1{|z|^{m+1}}\right) \qquad (m \in \Z_{\gqs 0})
\end{equation}
for certain coefficients $\tau_j$.
A computation shows
\begin{equation} \label{chum}
  F'(x) = \frac{1}{2x^{1/2}} - \frac{3}{4x^{2/2}} + \frac{3}{16x^{3/2}} + \frac{3}{16x^{4/2}} + \cdots + \frac{c_{R}}{x^{R/2}} + O\left( \frac{1}{x^{(R+1)/2}}\right),
\end{equation}
as $x \to \infty$. Use \e{ib2} to check that
$G'(z) = 1+ G(z)/z-G(z)^2$. Consequently, all the derivatives of $F(z)$ may be expressed using powers of $G(z)$.
 Lemma \ref{diff} now implies that the repeated derivatives of both sides of \e{chum} agree. Comparing \e{hjk} and \e{bop} shows
\begin{align}
  A(n) & = \frac{k'}{1!} F'(T) = \frac{k'}{2\sqrt{T}}+O\left( \frac 1{T}\right), \label{tt}\\
  -\delta(n)^2 & = \frac{(k')^{2}}{2!} F''(T) = -\frac{(k')^{2}}{8 T^{3/2}}+O\left( \frac 1{T^2}\right), \label{tt2}\\
g_r(n) & = \frac{(k')^{r}}{r!} F^{(r)}(T)
= O\left( \frac 1{n^{r-1/2}}\right) \qquad (r\gqs 3). \label{wro}
\end{align}
In terms of $n$ we obtain the estimates
\begin{gather}\label{adw}
  A(n)=\frac{\sqrt{k'}}{2\sqrt{n}}+O\left(\frac 1n\right), \qquad
  \delta(n)= \frac{(k')^{1/4}}{8^{1/2} n^{3/4}}\left(1+O\left(\frac 1{n^{1/2}}\right)\right), \\
  g_r(n)
= \binom{1/2}{r}\frac{\sqrt{k'}}{n^{r-1/2}} + O\left( \frac 1{n^{r}}\right). \label{est}
\end{gather}
The conditions of Definition \ref{dbf} and Theorems \ref{vog}, \ref{vog2} are seen to hold
and the theorem follows.
\end{proof}

Theorem 1.1 of  \cite{CP} claims that the sequence $b_k(n)$ is Hermite-Jensen (for the $J_{b_k}^{d,n}$ polynomials). There are several problems with the proof in \cite[Sect. 2]{CP},  the main one being that they try to use the misstated version of \cite[Thm. 3]{GORZ}. To establish that $b_k(n)$ is log-polynomial, their version of \e{bop} has
\begin{equation*}
  \log\left(\frac{b_k(n+j)}{b_k(n)}\right) = A^*(n) j - \delta^*(n)^2 j^2 + \sum_{r=3}^{m} g_r^*(n) j^r  +
O\left( \frac 1{n^{m+1/2}}\right)
\end{equation*}
for $g_3^*(n), \dots,g_m^*(n)$ all identically $0$. But these $g_r^*(n)$s cannot be $0$ since Proposition \ref{logp} and \e{not} demonstrate that $g_r^*(n)$ and $g_r(n)$, estimated in \e{est}, must agree up to the error $o(\delta(n)^{m+1})$ which becomes arbitrarily small for large $m$.
 Nevertheless, Theorem 1.1 of  \cite{CP} is  stated correctly because their expressions for $A^*(n)$ and $\delta^*(n)$ match ours in \e{tt}, \e{tt2}, \e{adw} to within the error allowed by Theorem \ref{vog2}.

\section{Gamma sequences} \label{last}
In this final section we see how an application of Theorems \ref{vog} and \ref{vog2} to the simple sequence $\alpha(n):=\G(n+\beta)$ of  gamma values leads naturally to the classical limit for generalized Laguerre polynomials
\begin{equation}\label{lagher}
   \left(\frac{2}{r}\right)^{d/2} L_d^{(r)}\left(r+\sqrt{2r} X\right) \to  \frac{H_d(-X)}{d!} \qquad \text{when} \quad r\to \infty,
\end{equation}
as in \cite[Eq. (6.2.14)]{AAR}. Set $L_d^{*(r)}(X):=X^d L_d^{(r)}(1/X)$, and we also obtain a reciprocal version of \e{lagher} that we have not found in the literature:
\begin{equation}\label{laghera}
   \left(2r\right)^{d/2} L_d^{*(r)}\left(\frac 1{r}+\frac{\sqrt{2}X}{r^{3/2}}\right) \to  \frac{H_d(X)}{d!}  \qquad \text{when} \quad r\to \infty.
\end{equation}

Start by putting  $n':=n+\beta$ where we may assume $0\lqs \beta \lqs 1$. Also set
$$
T:= n', \qquad \lambda:=j/n', \qquad F(z):=\log \G(z).
$$
Then
\begin{equation*}
  \log(\alpha(n+j))= F(T(1+\lambda)).
\end{equation*}
Let $m\gqs 2$ be a fixed integer. Using Stirling's formula to bound $F$, and Lemma \ref{tay}, yields the expansion
$$
 F(T(1+\lambda)) = \sum_{r=0}^{m} \frac{T^r}{r!} F^{(r)}(T) \cdot \lambda^r +O\left( n \log n|\lambda|^{m+1}\right)
$$
and hence, for $j$ in the  range  we want $1 \lqs j \lqs m+1$,
\begin{equation} \label{wimbg}
  \log\left(\frac{\alpha(n+j)}{\alpha(n)}\right) = \sum_{r=1}^{m} \frac{1}{r!} F^{(r)}(T) \cdot j^r +O\left( \frac{\log n}{n^{m}}\right)
\end{equation}
as $n \to \infty$ for an implied constant depending on $m$. 
The first derivative is the digamma function,
$$
F'(z)= \frac{\G'(z)}{\G(z)} = \psi(z),
$$
 with a well-known asymptotic expansion involving Bernoulli numbers \cite[6.3.18]{AS}:
\begin{equation*}
  F'(z)=  \psi(z)= \log z-\frac{1}{2z} - \sum_{j=2}^m \frac{B_j}{j z^j} + O\left( \frac{1}{|z|^{m+1}}\right).
\end{equation*}
Its derivatives are the polygamma functions and with \cite[6.4.11]{AS} they have the expansions for $r\gqs 2$
\begin{equation*}
  F^{(r)}(z)=\psi^{(r-1)}(z)= (-1)^r \left[ \frac{(r-2)!}{z^{r-1}}+  \frac{(r-1)!}{2z^{r}}+ \sum_{j=2}^m \frac{(j+r-2)! B_j}{j! z^{j+r-1}} + O\left( \frac{1}{|z|^{m+r}}\right)\right].
\end{equation*}
  From \e{wimbg} we obtain
\begin{align}
  A(n) & = \frac{1}{1!} F'(T) = \psi(T)=\log T+O\left( \frac 1{T}\right), \label{sfg}\\
  \delta(n)^2 & = \frac{1}{2!} F''(T) = \frac{\psi^{(1)}(T)}{2}=\frac{1}{2 T}+O\left( \frac 1{T^2}\right),\label{sfg2}\\
g_k(n) & = \frac{1}{k!} F^{(k)}(T) = \frac{\psi^{(k-1)}(T)}{k!} = O\left( \frac 1{n^{k-1}}\right) \qquad (k\gqs 3). \label{sfg3}
\end{align}
Therefore
\begin{gather}
  A(n) = \log(n+\beta)+O\left( \frac 1{n}\right), 
\qquad
 \delta(n)
=\frac{1}{\sqrt{2(n+\beta)}}\left(1+O\left( \frac 1{n}\right)\right). \label{sfg5}
\end{gather}
Checking the conditions of Definition \ref{dbf} and Theorems \ref{vog}, \ref{vog2}, we have proved

\begin{theorem}
The sequence $\alpha(n)=\G(n+\beta)$ is log-polynomial for
all degrees $m\gqs 2$ with data $\{\psi(n+\beta)$, $1$, $\left(\psi^{(1)}(n+\beta)/2\right)^{1/2}\}$. It is  Hermite-Jensen for this data as well as the simpler  $$\{ \log(n+\beta), 1, \left(2(n+\beta)\right)^{-1/2}\}.$$
\end{theorem}

The Jensen polynomials for the reciprocal sequence $1/\alpha(n)$ are essentially  the generalized  Laguerre polynomials we saw in \e{lagu}:
\begin{equation*}
  J^{d,n}_{1/\alpha}(X):=\sum_{j=0}^d \binom{d}{j} \frac{1}{\G(n+j+\beta)} X^j = \frac{d!}{\G(n+d+\beta)} L_d^{(n+\beta-1)}(-X).
\end{equation*}
This sequence is Hermite-Jensen for the data
$$\{ -\log(n+\beta), -1, \left(2(n+\beta)\right)^{-1/2}\}$$
by part (ii) of Proposition \ref{logp}, and for each $\beta$ we obtain
\begin{equation} \label{wei}
  \lim_{n\to \infty} (2(n+\beta))^{d/2}\frac{\G(n+\beta)}{\G(n+\beta+d)} L_d^{(n+\beta-1)}\left((n+\beta)\left(1-\frac{X}{\sqrt{2(n+\beta)}}\right)\right) = \frac{H_d(X/2)}{d!}.
\end{equation}

\begin{prop} \label{aabc}
When $r\to \infty$ we have
\begin{align}
   (2r)^{d/2}\frac{\G(r)}{\G(r+d)} L_d^{(r-1)}\left(r+\sqrt{2r}X\right) & \to \frac{H_d(-X)}{d!},
  \label{wa}\\
   (2r^3)^{d/2}\frac{\G(r)}{\G(r+d)} L_d^{*(r-1)}\left(\frac 1{r}+\frac{\sqrt{2}X}{r^{3/2}}\right) & \to \frac{H_d(X)}{d!}, \label{wb}
\end{align}
where $L_d^{*(r-1)}(x)$ is the reciprocal version of $L_d^{(r-1)}(x)$, defined after \e{lagher}.
\end{prop}
\begin{proof}
Let  $r=n+\beta$ and replace $X$ by $-2X$ on the left side of \e{wei} to get the left side of \e{wa}. It can be easily checked that the coefficient of $X^{d-j}$ in this expression is $r^{j/2}$ times a rational function of $r$. Therefore, if the limit \e{wa} holds for $r$ in the subsequence of integers it must hold for all $r \to \infty$. Similarly, for the reciprocal Jensen version with $K^{d,n}_{1/\alpha}$ in \e{hjak},  we obtain \e{wb}.
\end{proof}

It follows from Proposition \ref{aabc} that the zeros of $L_d^{(r)}(x)$ are real for $r$ sufficiently large. That is already clear in this case since $L_d^{(r)}(x)$ is orthogonal for $r\gqs -1$.

The identity
\begin{equation*}
  L_d^{(r-1)}(x) = L_d^{(r)}(x)- L_{d-1}^{(r)}(x)
\end{equation*}
from \cite[Eq. (22.7.30)]{AS} may be used to show that \e{wa} and \e{lagher} are really equivalent. For this, write
\begin{equation*}
  L_d^{(r)}\left(r+\sqrt{2r}X\right) = \sum_{j=0}^d \ell_{d,j}(r) X^j, \qquad \frac{H_d(-X)}{d!}  = \sum_{j=0}^d h_{d,j} X^j.
\end{equation*}
Then \e{wa} is true if and only if
\begin{equation} \label{jar}
  (2r)^{d/2}\frac{\G(r)}{\G(r+d)}\Bigr[ \ell_{d,j}(r) - \ell_{d-1,j}(r)\Bigl] \to h_{d,j}
\end{equation}
for $j=0,1,\dots,d$ as $r \to \infty$. Since
\begin{equation*}
  \frac{\G(r)}{\G(r+d)} = \frac{1}{(r+d-1)(r+d-2)\cdots (r)}= \frac{1}{r^d}\left( 1 + O\left( \frac{1}{r}\right)\right),
\end{equation*}
we see that \e{jar} is equivalent to
\begin{equation}  \label{jar2}
  \left(\frac{2}{r}\right)^{d/2} \ell_{d,j}(r) - \left(\frac{2}{r}\right)^{1/2} \left[ \left(\frac{2}{r}\right)^{(d-1)/2} \ell_{d-1,j}(r)\right] \to h_{d,j}
\end{equation}
for $j=0,1,\dots,d$.
Now, \e{jar2} being true for every positive integer $d$ is equivalent to
\begin{equation}  \label{jar3}
  \left(\frac{2}{r}\right)^{d/2} \ell_{d,j}(r)  \to h_{d,j} \qquad \text{for} \qquad j=0,1,\dots,d
\end{equation}
being true for every positive integer $d$ because the $d=0$ case of \e{jar3} is just $1 \to 1$. It follows that \e{wa} can be used to prove \e{lagher}, and vice versa.
In a similar way, \e{wb} and \e{laghera} are equivalent.

{\small
\bibliography{jen-bib}
}

{\small 
\vskip 5mm
\noindent
\textsc{Dept. of Math, The CUNY Graduate Center, 365 Fifth Avenue, New York, NY 10016-4309, U.S.A.}

\noindent
{\em E-mail address:} \texttt{cosullivan@gc.cuny.edu}
}

\end{document}